\documentclass{amsart}
\usepackage{amsmath}
\usepackage{amssymb}
\usepackage{amsthm}

\DeclareMathOperator{\divsymb}{div}

\DeclareMathOperator{\tr}{tr}

\DeclareMathOperator{\dvol}{dvol}

\DeclareMathOperator{\Ric}{Ric}

\DeclareMathOperator{\Real}{Re}

\newcommand{\oX}{\overline{X}}

\newcommand{\og}{\overline{g}}

\newcommand{\lp}{\langle}
\newcommand{\rp}{\rangle}
\newcommand{\lv}{\lvert}
\newcommand{\rv}{\rvert}

\newcommand{\mP}{\mathcal{P}}

\newcommand{\bC}{\mathbb{C}}

\newcommand{\bN}{\mathbb{N}}

\newcommand{\bR}{\mathbb{R}}

\def\sideremark#1{\ifvmode\leavevmode\fi\vadjust{\vbox to0pt{\vss
 \hbox to 0pt{\hskip\hsize\hskip1em
 \vbox{\hsize3cm\tiny\raggedright\pretolerance10000
 \noindent #1\hfill}\hss}\vbox to8pt{\vfil}\vss}}}

\newcommand{\comment}[1]{}

\newtheorem{thm}{Theorem}[section]
\newtheorem{prop}[thm]{Proposition}
\newtheorem{lem}[thm]{Lemma}
\newtheorem{cor}[thm]{Corollary}

\theoremstyle{definition}

\theoremstyle{remark}
\newtheorem{remark}[thm]{Remark}

\numberwithin{equation}{section}
\usepackage{esint}

\begin{document}

\title{On fractional GJMS operators}
\author{Jeffrey S. Case}
\thanks{JSC was partially supported by NSF Grant No.\ DMS-1004394}
\address{Department of Mathematics \\ Princeton University \\ Fine Hall, Washington Road \\ Princeton, NJ 08544}
\email{jscase@math.princeton.edu}
\author{Sun-Yung Alice Chang}
\thanks{SYAC was partially supported by NSF Grant No.\ DMS-1104536}
\address{Department of Mathematics \\ Princeton University \\ Fine Hall, Washington Road \\ Princeton, NJ 08544}
\email{chang@math.princeton.edu}
\date{\today}
\begin{abstract}
We describe a new interpretation of the fractional GJMS operators as generalized Dirichlet-to-Neumann operators associated to weighted GJMS operators on naturally associated smooth metric measure spaces.  This gives a geometric interpretation of the Caffarelli--Silvestre extension for $(-\Delta)^\gamma$ when $\gamma\in(0,1)$, and both a geometric interpretation and a curved analogue of the higher order extension found by R.\ Yang for $(-\Delta)^\gamma$ when $\gamma>1$.  We give three applications of this correspondence.  First, we exhibit some energy identities for the fractional GJMS operators in terms of energies in the compactified Poincar\'e--Einstein manifold, including an interpretation as a renormalized energy.  Second, for $\gamma\in(1,2)$, we show that if the scalar curvature and the fractional $Q$-curvature $Q_{2\gamma}$ of the boundary are nonnegative, then the fractional GJMS operator $P_{2\gamma}$ is nonnegative.  Third, by assuming additionally that $Q_{2\gamma}$ is not identically zero, we show that $P_{2\gamma}$ satisfies a strong maximum principle.
\end{abstract}
\maketitle

%%%%%%%%%%%%%%%%%%%%%%%%%%%%%%%%%%%%%%%%%%%%%%%%%%%%%%%%%%%%%%%%%
%                                                               %
% Structure of the document                                     %
%                                                               %
% 1. Intro                                                      %
% *. Acknowledgments                                            %
% 2. Smooth metric measure spaces                               %
% 3. The higher order extension                                 %
% 4. Renormalized energy                                        %
%                                                               %
%%%%%%%%%%%%%%%%%%%%%%%%%%%%%%%%%%%%%%%%%%%%%%%%%%%%%%%%%%%%%%%%%

\section{Introduction}
\label{sec:intro}

There has recently been a great deal of interest in fractional-order nonlocal operators and scalar invariants related to them; e.g.\ \cite{CaffarelliSilvestre2007,ChangGonzalez2011,GonzalezQing2010}.  In this article we are primarily interested in the fractional-order conformally covariant powers of the Laplacian (henceforth fractional GJMS operators) and their associated fractional $Q$-curvatures as introduced by Graham and Zworski~\cite{GrahamZworski2003}.  Unlike familiar geometric objects such as the scalar curvature, the mean curvature, and the conformal Laplacian, it is still not entirely clear what these objects are geometrically.  In this paper, we provide a geometric interpretation of the fractional GJMS operators.  Specifically, we exhibit a natural correspondence between the fractional GJMS operators on the boundary of a Poincar\'e--Einstein manifold and weighted GJMS operators on a natural smooth metric measure space obtained by compactifying the Poincar\'e--Einstein manifold.  This generalizes the relationship between the conformal half-Laplacian and the conformal Laplacian which was studied in depth by Escobar~\cite{Escobar1992a,Escobar1992}.  We give three applications of this correspondence.  First, we show that the energies of the fractional GJMS operators can be naturally interpreted as renormalized energies of a natural second-order operator, generalizing a result of R.\ Yang~\cite{Yang2013} in the flat case.  Second, we give sufficient conditions in terms of the conformal boundary for the fractional GJMS operators with $\gamma\leq2$ to be nonnegative.  Third, we give sufficient conditions in terms of the conformal boundary for the fractional GJMS operators to satisfy a strong maximum principle.  The latter two applications partially arise by exploiting the relationship between the fractional GJMS operators $P_{2\gamma}$ of order $2\gamma\in(2,4)$ and the weighted Paneitz operator and considering the weighted analogues of positivity results known for the Paneitz operator~\cite{Gursky1999,GurskyMalchiodi2014,XuYang2001}. Third, we use the relationship between the fractional GJMS operators $P_{2\gamma}$ of order $2\gamma\in(2,4)$ and the weighted Paneitz operator to give sufficient conditions for the positivity of and strong maximum principles for $P_{2\gamma}$ which are analogous to conditions known for the Paneitz operator~\cite{Gursky1999,GurskyMalchiodi2014,XuYang2001}.

The simplest case of a fractional-order conformally covariant power of the Laplacian arises in connection to the operator $B=\frac{\partial}{\partial\eta}+\frac{n-1}{2n}H$ defined on the boundary $(M^n,h)$ of a compact Riemannian manifold-with-boundary $(\oX^{n+1},g)$, where $\frac{\partial}{\partial\eta}$ is the outward pointing normal and $H$ is the mean curvature of the boundary $(M^n,h)$.  This operator is conformally covariant in the sense that if $\sigma\in C^\infty(\oX)$, then the operator $\hat B$ defined in terms of the conformally rescaled metric $\hat g:=e^{2\sigma}g$ is related to $B$ by the conjugation
\[ \hat B = e^{-\frac{n+1}{2}\sigma}\circ B\circ e^{\frac{n-1}{2}\sigma}, \]
where the right hand side is to be interpreted as the pre- and post-composition of $B$ with multiplication operators.  As pointed out and put to great effect by Escobar~\cite{Escobar1992a,Escobar1992} in his study of the Yamabe Problem on manifolds with boundary, the operator $B$ should be regarded as the boundary operator associated to the conformal Laplacian $L_2:=-\Delta_{g}+\frac{n-1}{4n}R_{g}$ defined in the interior $(\oX^{n+1},g)$.  The conformal Laplacian is also conformally covariant; in terms of $\hat g=e^{2\sigma}g$ we have that
\[ \hat L_2 = e^{-\frac{n+3}{2}\sigma}\circ L_2\circ e^{\frac{n-1}{2}\sigma} . \]
The conformal covariance of both operators implies that the sums
\begin{multline}
\label{eqn:escobar}
\int_{X} U\,L_2U\,\dvol_{g} + \int_M f\,Bf\,\dvol_h \\ = \int_{X} \left[\lv\nabla U\rv_{g}^2 + \frac{n-1}{4n}R_{g}U^2\right]\dvol_{g} + \frac{n-1}{2n}\int_M Hf^2\,\dvol_h
\end{multline}
are conformally covariant for any $U\in W^{1,2}(X)$ with $f=U\rv_M$; here the equality follows by integration by parts, and is the reason we declare $B$ to be the boundary operator associated to $L_2$.  In particular, solving for the unique extension $U$ of $f$ such that $L_2U=0$ yields the energy identity
\begin{equation}
\label{eqn:escobar_energy}
\int_M f\,Bf\,\dvol_h = \int_{X} \left[\lv\nabla U\rv_{g}^2 + \frac{n-1}{4n}R_{g}U^2\right]\dvol_{g} + \frac{n-1}{2n}\int_M Hf^2\,\dvol_h .
\end{equation}

In the flat case $(\bR^n,dx^2)$ as the boundary of $(\bR_+^{n+1},dx^2\oplus dy^2)$, the above-described procedure is the classical method for defining $(-\Delta)^{1/2}$.  More precisely, given $f\in C^\infty(\bR^n)$, one can define the half-Laplacian by $(-\Delta)^{1/2}f=-\frac{\partial U}{\partial y}$ for $U$ the harmonic function in $(\bR_+^{n+1},dx^2\oplus dy^2)$ such that $U(\cdot,0)=f$; see~\cite{CaffarelliSilvestre2007} for details.  For this reason and due to its conformal covariance, it is thus natural to consider $B$ to be the conformal half-Laplacian (cf.\ \cite{GuillarmouGuillope2007}).

An important feature of the realization of the (conformal) half-Laplacian as the Dirichlet-to-Neumann operator associated to the (conformal) Laplacian is that it allows one to derive estimates for the former operator, which is a nonlocal operator, using well-established techniques developed for the latter operator.  This motivated Caffarelli and Silvestre~\cite{CaffarelliSilvestre2007} to identify the fractional Laplacian $(-\Delta)^\gamma$ as the Dirichlet-to-Neumann operator associated to a second-order degenerate elliptic operator in the interior.  More precisely, they showed that for each $\gamma\in(0,1)$ there is an explicit constant $c_\gamma$ such that for any $f\in C^\infty(\bR^n)$, it holds that
\begin{equation}
\label{eqn:intro_caffarelli_silvestre}
(-\Delta)^\gamma f = c_\gamma \lim_{y\to0} y^{m_0}\frac{\partial U}{\partial y}
\end{equation}
for $m_0=1-2\gamma$ and $U\in C^\infty(\bR_+^{n+1})\cap C^0(\overline{\bR_+^{n+1}})$ such that $U(\cdot,0)=f$ and $\divsymb\left(y^{m_0}\nabla U\right)=0$.  Their analysis is greatly assisted by the curious observation that one can formally think of the condition $\divsymb\left(y^{m_0}\nabla U\right)=0$ as the condition that $U$ be a harmonic function in $\bR^{n+1+m_0}$ which is radially symmetric with respect to the $\bR^{m_0}$ factor.

As shown by Chang and Gonz\'alez~\cite{ChangGonzalez2011}, the Caffarelli--Silvestre extension is a reformulation of the definition of the fractional GJMS operators as introduced by Graham and Zworski~\cite{GrahamZworski2003}.  Given a Poincar\'e--Einstein manifold $(X^{n+1},M^n,g_+)$, Graham and Zworski showed that for all but a finite number of values $\gamma\in(0,n/2)$ one can define, via scattering theory, a formally self-adjoint pseudodifferential operator $P_{2\gamma}$ on the boundary $(M^n,h)$ such that the symbol of $P_{2\gamma}$ is the symbol of $(-\Delta)^\gamma$ and $P_{2\gamma}$ is conformally covariant.  Moreover, in the case $\gamma\in\bN$ the operators $P_{2\gamma}$ agree with the GJMS operators $P_{2k}$ introduced by the eponymous Graham, Jenne, Mason and Sparling~\cite{GJMS1992} via the ambient metric; in particular, $P_{2k}$ is independent of the Poincar\'e--Einstein fill-in.  For $\gamma\in(0,1)$, Chang and Gonz\'alez~\cite{ChangGonzalez2011} observed that, when written in terms of the compactification $g:=r^2g_+$, where $r$ is the geodesic defining function associated to $h$, the scattering definition of $P_{2\gamma}$ becomes
\begin{equation}
\label{eqn:intro_chang_gonzalez}
P_{2\gamma}f = c_\gamma \lim_{\gamma\to0}r^{m_0}\frac{\partial U}{\partial r}
\end{equation}
for $m_0$ and $c_\gamma$ as in~\eqref{eqn:intro_caffarelli_silvestre} and $U$ the function in $(\oX,g)$ such that $U\rv_M=f$ and
\begin{equation}
\label{eqn:chang_gonzalez_2nd_order}
\divsymb\left(r^{m_0}\nabla U\right) + E(r,m_0)U = 0 .
\end{equation}
While a clear geometric interpretation of the error term $E(r,m_0)$ is not given in~\cite{ChangGonzalez2011}, it was observed that it vanishes in the model case $\left(\bR_+^{n+1},\bR^n,y^{-2}(dx^2\oplus dy^2)\right)$ and that $E(r,0)=\frac{n-1}{4n}R$; i.e.\ the operator on the left hand side of~\eqref{eqn:chang_gonzalez_2nd_order} is the conformal Laplacian when $m_0=0$, recovering the relationship between $L_2$ and $B$ exposed by Escobar~\cite{Escobar1992a,Escobar1992}.  Note that, in comparison to~\eqref{eqn:escobar_energy}, it follows from~\eqref{eqn:intro_chang_gonzalez} and~\eqref{eqn:chang_gonzalez_2nd_order} that
\begin{equation}
\label{eqn:intro_chang_gonzalez_energy}
\int_M f\,P_{2\gamma}f\,\dvol_g = \int_X \left[\lv\nabla U\rv_{\og}^2 + E(r,m_0)U^2\right]r^{m_0}\dvol_{\og}
\end{equation}
for $\gamma\in(0,1)$; that we do not see the analogue of the mean curvature term in~\eqref{eqn:escobar_energy} is due to the assumptions that $(X^{n+1},M^n,g_+)$ is Poincar\'e--Einstein and $r$ is a geodesic defining function (cf.\ \cite{ChangGonzalez2011} and Section~\ref{sec:extension}).

These results still leave some questions remaining.  For example, what happens in the cases $\gamma\in(1,n/2)$?  Is there some geometric interpretation of the operator appearing in~\eqref{eqn:chang_gonzalez_2nd_order}?  Can this explain the formal dimensional interpretation of the parameter $m_0$ used by Caffarelli and Silvestre~\cite{CaffarelliSilvestre2007}?  In this article, we answer these questions by observing that~\eqref{eqn:chang_gonzalez_2nd_order} can be identified as the weighted conformal Laplacian on a canonical smooth metric measure space and that there is a higher order version of the definition of the fractional GJMS operators via extensions, as we now explain.

In the flat case, the first question was answered by R.\ Yang~\cite{Yang2013}.  A key issue to address is that the right hand side of~\eqref{eqn:intro_chang_gonzalez_energy} will always be infinite when $\gamma>1$, and so one must find a new energy identity for the fractional Laplacian of order greater than two.  R.\ Yang overcame this issue by showing that for $k=\lfloor\gamma\rfloor$ and $m_k:=2k+1-2\gamma\in(-1,1)$, one can define the operator $(-\Delta)^\gamma$ on $(\bR^n,dx^2)$ by
\[ (-\Delta)^\gamma f = c_\gamma\lim_{y\to 0} y^{m_k}\frac{\partial}{\partial y}(\Delta_{m_k})^kU , \]
where $\Delta_{m_k}=\Delta + m_ky^{-1}\partial_y$ and $U$ is the extension of $f$ such that $(\Delta_{m_k})^{k+1}U=0$; see Theorem~\ref{thm:yang} below for the precise statement.  The key point is that this definition yields the energy identity
\begin{equation}
\label{eqn:intro_yang}
\int_{\bR^n} f\,(-\Delta)^\gamma f = c_\gamma \int_{\bR_+^{n+1}} \lv\Delta_{m_k}^{(k+1)/2}U\rv^2\,y^{m_k}\dvol ,
\end{equation}
where we declare $\Delta_{m_k}^{(k+1)/2}U=\nabla(\Delta_{m_k})^{k/2}U$ if $k$ is even.

A key property of the operator $\Delta_{m_k}$ is that it is formally self-adjoint with respect to the measure $y^{m_k}\dvol$ on $(\bR_+^{n+1},dx^2\oplus dy^2)$.  Combining the ideas underlying Escobar's energy identity~\eqref{eqn:escobar_energy} and R.\ Yang's energy identity~\eqref{eqn:intro_yang}, it is natural to expect that the curved analogue of the right hand side of~\eqref{eqn:intro_yang} should be the energy associated to a conformally covariant operator of order $k+1$ defined on a smooth metric measure space; i.e.\ a Riemannian manifold with measure.  Meaningful notions of the conformal Laplacian and the Paneitz operator --- conformally covariant operators of order two and four, respectively --- on smooth metric measure spaces were given in~\cite{Case2011t}, and they are precisely the operators which appear in the curved analogue of~\eqref{eqn:intro_yang}.

In the most general sense studied in~\cite{Case2011t}, a smooth metric measure space is a four-tuple $(X^{n+1},g,v^m\dvol,\mu)$ of a Riemannian manifold with (possibly empty) boundary, a nonnegative function $v\in C^\infty(X)$ such that $v^{-1}(0)=\partial X$, and parameters $m\in\bR\cup\{\pm\infty\}$ and $\mu\in\bR$.  Heuristically, such a space should be thought of as the base of the warped product $(X^{n+1}\times F_\mu^m,g\oplus v^2g_F)$ for $(F_\mu^m,g_F)$ the $m$-dimensional simply-connected spaceform with $\Ric(g_F)=\mu g_F$.  In this article we are exclusively interested in the case $\mu=m-1$, so that $(F_\mu^m,g_F)$ is the $m$-dimensional sphere with constant sectional curvature one and the smooth metric measure space $(\bR_+^{n+1},dx^2\oplus dy^2,y^m\dvol,m-1)$ is formally the base of flat Euclidean $(m+n+1)$-dimensional space, precisely the heuristic used by Caffarelli and Silvestre~\cite{CaffarelliSilvestre2007}.

The point of working with smooth metric measure spaces is that they provide a natural geometric perspective on the Poisson equations appearing in the work of Graham and Zworski~\cite{GrahamZworski2003}.  In Section~\ref{sec:extension} we show that if $(X^{n+1},M^n,g_+)$ is a Poincar\'e--Einstein manifold, then for a given $\gamma\in(0,n/2)$, the Poisson equation which gives rise to the operator $P_{2\gamma}$ is precisely the weighted conformal Laplacian of $(X^{n+1},g_+,1^{m_0}\dvol,m_0-1)$ for $m_0=1-2\gamma$.  Since the weighted conformal Laplacian is conformally covariant, this seamlessly allows one to rewrite the Poisson equation in terms of compactified metrics, and in particular identifies the operator appearing in~\eqref{eqn:intro_chang_gonzalez} as the weighted conformal Laplacian of the compactified metric.  Moreover, since smooth metric measure spaces with $m=0$ are Riemannian manifolds, we again recover the connection to Escobar's work on the boundary Yamabe Problem.  One can also define weighted analogues of the GJMS operators~\cite{GJMS1992} which, analogous to the GJMS operators for Einstein manifolds~\cite{FeffermanGraham2012,Gover2006q}, admit factorizations for the smooth metric measure spaces $(X^{n+1},g_+,1^m\dvol,m-1)$ associated to Poincar\'e--Einstein manifolds.  As a consequence, we derive the curved analogue of the extension found by R.\ Yang~\cite{Yang2013} for the fractional GJMS operators $P_{2\gamma}$ for all $\gamma\in(0,n/2)$.  When $\gamma\in(0,1)$, this gives a more geometric interpretation of the result of Chang and Gonz\'alez~\cite{ChangGonzalez2011}; see Theorem~\ref{thm:01_case} for the statement.  When $\gamma\in(1,2)$, a special case of our work is the following result.

\begin{thm}
\label{thm:intro_12_case}
Let $(X^{n+1},M^n,g_+)$ be a Poincar\'e--Einstein manifold, fix a representative $h$ of the conformal boundary, and let $r$ be the geodesic defining function associated to $h$.  Let $\gamma\in(1,2)$, set $m_1=3-2\gamma$ and $s=\frac{n}{2}+\gamma$, and consider the smooth metric measure space
\[ \left(\oX^{n+1},g:=r^2g_+,r^{m_1}\dvol,m_1-1\right) . \]
Given $f\in C^\infty(M)$, the function $U$ is the solution of the boundary value problem
\begin{equation}
\label{eqn:intro_12_case_boundary}
\begin{cases}
L_{4,\phi_1}^{m_1} U = 0, & \quad\text{in $X$}, \\
U = f, & \quad\text{on $M$}, \\
\lim_{r\to0}r^{m_1}\frac{\partial U}{\partial r} = 0
\end{cases}
\end{equation}
if and only if the function $u=r^{n-s}U$ is the solution of the Poisson problem
\begin{equation}
\label{eqn:intro_12_case_poisson}
\begin{cases}
   -\Delta_{g_+}u - s(n-s)u = 0, & \quad\text{in $X$}, \\
   u = Fr^{n-s} + Gr^s, & \quad F,G\in C^\infty(\oX) \\
   F\rv_{r=0} = f .
\end{cases}
\end{equation}
Moreover, the solution $U$ of~\eqref{eqn:intro_12_case_boundary} is such that
\begin{equation}
\label{eqn:intro_12_case_p2gamma}
P_{2\gamma}f = \frac{d_\gamma}{8\gamma(\gamma-1)}\lim_{r\to0} r^{m_1}\frac{\partial}{\partial r}\Delta_{\phi_1}U .
\end{equation}
\end{thm}

Here $d_\gamma$ is an explicit constant given in~\eqref{eqn:scattering_definition} and $L_{2,\phi}^m$ and $L_{4,\phi}^m$ are the weighted conformal Laplacian and the weighted Paneitz operator, respectively, introduced in~\cite{Case2011t}; see Section~\ref{sec:smms} for their definitions.  Theorem~\ref{thm:12_case} is a more general result which allows great freedom in the choice of defining function.  Theorem~\ref{thm:intro_12_case} has two useful consequences.  First, it shows that solutions to the fourth order boundary value problem~\eqref{eqn:intro_12_case_boundary} are in one-to-one correspondence with solutions to the second order Poisson equation~\eqref{eqn:intro_12_case_poisson}.  Second, the formula~\eqref{eqn:intro_12_case_p2gamma} readily allows one to derive an energy identity for $P_{2\gamma}$ in terms of the energy of the weighted Paneitz operator which generalizes~\eqref{eqn:intro_yang}; see Corollary~\ref{cor:12_inequality} for details.

We expect that this new geometric relationship between fractional GJMS operators on boundaries of Poincar\'e--Einstein manifolds and weighted GJMS operators in their compactifications will lead to many new insights into the nature of the fractional GJMS operators.  In this article we give three such applications.

The first application is again a curved analogue of an observation made by R.\ Yang~\cite{Yang2013}.  For $\gamma>1$, the right hand side of~\eqref{eqn:intro_chang_gonzalez_energy} is infinite.  However, R.\ Yang showed that the energy of the fractional Laplacian $(-\Delta)^\gamma$ can be recovered as a renormalization of the right hand side of~\eqref{eqn:intro_chang_gonzalez_energy}.  His proof immediately carries over to the curved setting using the weighted conformal Laplacian and the weighted Paneitz operator.  This inspired Theorem~\ref{thm:12_renormalized_energy}, though the proof we present is greatly simplified by using the asymptotics of solutions to the Poisson equation~\eqref{eqn:poisson_equation}; for details, see Section~\ref{sec:energy}.

The second application is to give natural sufficient conditions for the positivity of the fractional GJMS operators of order $2\gamma\in(2,4)$.  For the Paneitz operator, it is known through work of Gursky~\cite{Gursky1999} (in dimension four), Xu and P.\ Yang~\cite{XuYang2001} (in dimensions six and larger) and Gursky and Malchiodi~\cite{GurskyMalchiodi2014} (in dimensions at least five) that if a Riemannian manifold has nonnegative scalar curvature and nonnegative (fourth-order) $Q$-curvature, then the Paneitz operator is nonnegative, and moreover, there is rigidity when the kernel of the Paneitz operator is nontrivial.  Here we establish the corresponding result for $P_{2\gamma}$ with $\gamma\in(1,2)$ in terms of the scalar curvature and the fractional $Q$-curvature $Q_{2\gamma}$ of the boundary.

\begin{thm}
\label{thm:positive}
Let $(X^{n+1},M^n,g_+)$ be a Poincar\'e--Einstein manifold, let $\gamma\in(1,2)$ if $n\geq 4$ and let $\gamma\in(1,3/2]$ if $n=3$.  Suppose that there is a representative of the conformal boundary with nonnegative scalar curvature and nonnegative fractional $Q$-curvature $Q_{2\gamma}$.  Then $P_{2\gamma}\geq0$.  Moreover, $\ker P_{2\gamma}\not=\{0\}$ if and only if $Q_{2\gamma}\equiv0$ or $n=2\gamma=3$, in which case $\ker P_{2\gamma}=\bR$ consists of the constant functions.
\end{thm}

Note in particular that this result applies to the critical third-order operator $P_3$ in the $3+1$ setting of Poincar\'e--Einstein manifolds.

The third application is also a fractional-order analogue of work of Gursky and Malchiodi~\cite{GurskyMalchiodi2014}.  One difficulty in studying questions involving the Paneitz operator is that, being a fourth-order operator, it does not in general satisfy a maximum principle.  Gursky and Malchiodi showed that when the underlying metric has nonnegative scalar curvature and semi-positive $Q$-curvature --- i.e.\ $Q\geq0$ and $Q\not\equiv0$ --- the Paneitz operator does satisfy a strong maximum principle.  By adapting the ideas from~\cite{GurskyMalchiodi2014}, we show that the fractional GJMS operators $P_{2\gamma}$ for $\gamma\in(1,2)$ likewise satisfy a strong maximum principle when the representative $h$ of the conformal boundary has nonnegative scalar curvature and $Q_{2\gamma}$ semi-positive.

\begin{thm}
\label{thm:maximum}
Let $(X^{n+1},M^n,g_+)$ be a Poincar\'e--Einstein manifold and suppose that there is a representative $h$ for the conformal boundary with scalar curvature $R\geq0$ and semi-positive fractional $Q$-curvature $Q_{2\gamma}$ for some $1<\gamma<\min\{2,n/2\}$ fixed.  Then for any $f\in C^\infty(M)$ such that $P_{2\gamma}f\geq0$, either $f>0$ or $f\equiv0$.  Moreover, if $f>0$, then the representative $f^{\frac{4}{n-2\gamma}}h$ of the conformal boundary has positive scalar curvature and nonnegative fractional $Q$-curvature $Q_{2\gamma}$.
\end{thm}

A similar statement holds for the fractional GJMS operators $P_{2\gamma}$ with $\gamma\in(0,1)$, in which case one only needs to assume that the representative of the conformal boundary has semi-positive fractional $Q$-curvature.  The main new ingredient in the proof of Theorem~\ref{thm:maximum} relative to the arguments of Gursky and Malchiodi~\cite{GurskyMalchiodi2014} is an improved, relative to the statement given by Gonz\'alez and Qing~\cite{GonzalezQing2010}, Hopf Lemma for degenerate elliptic operators on smooth metric measure spaces.

We remark here that Hang and P.\ Yang~\cite{HangYang2014} have recently taken another perspective on the strong maximum principle for the Paneitz operator, and in particular showed that one only need to assume that $Q$ is semi-positive and the Yamabe constant is positive to recover the result of Gursky and Malchiodi~\cite{GurskyMalchiodi2014}.  A key idea in their work is to use the natural power of the Green's function $G$ for the conformal Laplacian and as an approximate Green's function for the Paneitz operator while exploiting conformal covariance to easily compute the error terms.  Similar ideas would work for the fractional GJMS operators $P_{2\gamma}$ with $\gamma\in(1,2)$, though we have opted not to pursue the details in this article. 

The basic idea underlying both Theorem~\ref{thm:positive} and Theorem~\ref{thm:maximum} is to find a special choice of defining function which naturally encodes the fractional $Q$-curvature in the interior geometry of the compactified smooth metric measure space.  This idea has its origins in the work of Lee~\cite{Lee1995} (see also~\cite{Qing2003}), where a relationship between the scalar curvature of (a representative of) the conformal boundary and the scalar curvature of a particular conformal compactification of a Poincar\'e--Einstein manifold was established.  Lee's defining function is obtained as a particular solution of the Poisson equation~\eqref{eqn:poisson_equation}.  Building on this idea, we will use the special defining function $\rho^\ast$ introduced in~\cite{ChangGonzalez2011,GonzalezQing2010} to see the influence of the fractional $Q$-curvature in the interior geometry of what we will call the adapted smooth metric measure space; i.e.\ the smooth metric measure space arising by using $\rho^\ast$ as the defining function in.  In particular, $Q_{2\gamma}$ naturally arises in the asymptotic expansion of the scalar curvature of the adapted metric near the boundary.  We exploit this fact in conjunction with our observation that, as a generalization of~\cite{Lee1995}, if $\gamma>1$ and the representative of the conformal boundary has nonnegative scalar curvature, then the adapted metric has positive scalar curvature in the interior.

Throughout this article we consider Poincar\'e--Einstein manifolds $(X^{n+1},M^n,g_+)$ in the strong sense that $\Ric(g_+)=-ng_+$ globally.  This is used in a strong way via Theorem~\ref{thm:weighted_gjms_factorization}, which realizes the Poisson operators which give rise to the fractional GJMS operators as factors of the associated weighted GJMS operators.  The connection to smooth metric measure spaces allows us to easily understand the geometry of the interior, which in particular allows us to adapt the geometric ideas from~\cite{Gursky1999,GurskyMalchiodi2014,XuYang2001} to our setting.  It is this latter point which is the most important.  Theorem~\ref{thm:intro_12_case} relies only on conformal covariance, and so one could simply compute as in~\cite{ChangGonzalez2011} to establish a more general form of Theorem~\ref{thm:intro_12_case} for Poincar\'e metrics as in~\cite{FeffermanGraham2012} or even asymptotically hyperbolic manifolds.  However, it is unclear to us whether Theorem~\ref{thm:positive} and Theorem~\ref{thm:maximum} will remain true in those settings, or if one will require some additional assumptions (cf.\ \cite{GuillarmouQing2010}).

This article is organized as follows.  The first two sections review the necessary background for this article:  Section~\ref{sec:scattering} contains the relevant facts of the scattering approach to defining the fractional GJMS operators and Section~\ref{sec:smms} contains the relevant definitions and facts about smooth metric measure spaces necessary to realize the fractional GJMS operators as boundary operators associated to weighted GJMS operators; note that Theorem~\ref{thm:weighted_gjms_factorization} is new, giving a factorization of the weighted GJMS operators on a distinguished class of quasi-Einstein smooth metric measure spaces.  The remaining sections detail our results and techniques.  In Section~\ref{sec:extension} we describe how to define the fractional GJMS operators in terms of an extension problem involving weighted GJMS operators on suitable smooth metric measure spaces, which as a corollary yields a relationship between the energy of the fractional GJMS operators and the energy of the corresponding weighted GJMS operators.  In Section~\ref{sec:energy} we give another interpretation of the energy of the fractional GJMS operators as the finite part of the energy of the second-order extension operator introduced by Chang and Gonz\'alez~\cite{ChangGonzalez2011}.  In Section~\ref{sec:adapted} we detail the construction of the adapted smooth metric measure spaces and discuss their properties; in particular, we prove the aforementioned result on the nonnegativity of the scalar curvature of the adapted metric.  In Section~\ref{sec:positivity} we prove Theorem~\ref{thm:positive} and Theorem~\ref{thm:maximum}, as well as other related positivity results for the fractional GJMS operators; this section also includes the aforementioned Hopf lemma.  We also include an appendix which gives a partial factorization of the GJMS operators on normalized products of a positively-curved Einstein manifold with a negatively-curved Einstein manifold which is needed to prove Theorem~\ref{thm:weighted_gjms_factorization}.

\subsection*{Acknowledgments} The authors would like to thank Mar\'ia del Mar Gonz\'alez, Robin Graham, Jie Qing, Paul Yang, and Ray Yang for helpful conversations while this paper was being put together.

\section{Fractional GJMS operators via scattering theory}
\label{sec:scattering}

In this section we recall the definition of the fractional GJMS operators via scattering theory~\cite{GrahamZworski2003}.  A triple $(X^{n+1},M^n,g_+)$ is a \emph{Poincar\'e--Einstein manifold} if
\begin{enumerate}
\item $X^{n+1}$ is (diffeomorphic to) the interior of a compact manifold $\oX^{n+1}$ with boundary $\partial\oX=M^n$,
\item $(X^{n+1},g_+)$ is complete with $\Ric(g_+)=-ng_+$, and
\item there exists a nonnegative $\rho\in C^\infty(X)$ such that $\rho^{-1}(0)=M^n$, $d\rho\not=0$ along $M$, and the metric $g:=\rho^2g_+$ extends to a smooth metric on $\oX^{n+1}$.
\end{enumerate}
A function $\rho$ satisfying these properties is called a \emph{defining function}, though we will later allow defining functions to have less regularity at the boundary $M$ (cf.\ Theorem~\ref{thm:01_case}).  Since $\rho$ is only determined up to multiplication by a positive smooth function on $\oX$, it is clear that only the conformal class $[h]:=[g|_{TM}]$ on $M$ is well-defined for a Poincar\'e--Einstein manifold.  We call the pair $(M^n,[h])$ the \emph{conformal boundary} of the Poincar\'e--Einstein manifold $(X^{n+1},M^n,g_+)$ and we call a metric $h\in[h]$ a \emph{representative of the conformal boundary}.

Given a Poincar\'e--Einstein manifold $(X^{n+1},M^n,g_+)$ and a representative $h$ of the conformal boundary, there exists a unique defining function $r$, called the \emph{geodesic defining function}, such that, locally near $M$, the metric $g_+$ takes the form
\[ g_+ = r^{-2}\left(dr^2 + h_r\right) \]
for $h_r$ a one-parameter family of metrics on $M$ with $h_0=h$ and having an asymptotic expansion which is even in powers of $r$, at least up to order $n$ (cf.\ \cite{FeffermanGraham2012,GrahamLee1991,Lee1995}).  For simplicity, we assume everywhere except Corollary~\ref{cor:signs_of_r_and_q2} that $h$ is even to all orders; since we restrict our attention to $\gamma<\frac{n}{2}$, this assumption plays no role except to simply a few statements in this section.

It is well-known (see~\cite{GrahamZworski2003,MazzeoMelrose1987} for more general statements) that given $f\in C^\infty(M)$ and $s\in\bC$ such that $\Real s>\frac{n}{2}$, $s\not\in\frac{n}{2}+\bN$, and $s(n-s)$ is not in the pure-point spectrum $\sigma_{\mathrm{pp}}(-\Delta_{g_+})$ of $-\Delta_{g_+}$, the Poisson equation
\begin{equation}
\label{eqn:poisson_equation}
-\Delta_{g_+}u - s(n-s)u = 0 \qquad\text{in $X$}
\end{equation}
has a unique solution of the form
\begin{equation}
\label{eqn:poisson_asymptotics}
u = Fr^{n-s} + Hr^s \qquad\text{for $F,H\in C^\infty\left(\oX\right)$ and $F\rv_M=f$}.
\end{equation}
Indeed, $F$ has an asymptotic expansion
\begin{equation}
\label{eqn:F_expansion}
F = f_{(0)} + f_{(2)}r^2 + f_{(4)}r^4 + \dotsb
\end{equation}
for $f_{(0)}=f$ and all the functions $f_{(2\ell)}$ are determined by $f$.  The \emph{Poisson operator} $\mP(s)$ is the operator which maps $f$ to the solution $u=\mP(s)f$, and this operator is analytic for $s(n-s)\not\in\sigma_{\mathrm{pp}}(-\Delta_{g_+})$.  The \emph{scattering operator} is defined by $S(s)f = H\rv_M$.  This defines a meromorphic family of pseudodifferential operators in $\Real(s)>\frac{n}{2}$.  The values $s\in\left\{\frac{n}{2}+1,\frac{n}{2}+2,\dotsc\right\}$ are simple poles, and are known as the trivial poles.  The scattering operator may have other poles, but we shall assume for the remainder of this article that we are not in these exceptional cases.

Given $\gamma\in\left(0,\frac{n}{2}\right)$, Graham and Zworski~\cite{GrahamZworski2003} defined the \emph{fractional GJMS operator $P_{2\gamma}$} as the operator
\begin{equation}
\label{eqn:scattering_definition}
P_{2\gamma}f := d_\gamma S\left(\frac{n}{2}+\gamma\right)f \qquad\text{for $d_\gamma=2^{2\gamma}\frac{\Gamma(\gamma)}{\Gamma(-\gamma)}$}.
\end{equation}
For $\gamma\in\bN$, this definition recovers the GJMS operators~\cite{GJMS1992}.  Due to the overuse of the symbol `$P$', we usually denote the GJMS operators by $L_{2k}$ when $k\in\bN$.  From its definition, it is clear that $P_{2\gamma}$ is linear.  For $\gamma\in\left(0,\frac{n}{2}\right)$, Graham and Zworski showed that $P_{2\gamma}$ is a formally self-adjoint pseudodifferential operator with principle symbol equal to the principle symbol of $(-\Delta)^\gamma$, and moreover, if $\hat h=e^{2\sigma}h$ is another choice of conformal representative of the conformal boundary, then
\[ \hat P_{2\gamma}f = e^{-\frac{n+2\gamma}{2}\sigma} P_{2\gamma}\left(e^{\frac{n-2\gamma}{2}\sigma}f\right) \]
for all $f\in C^\infty(M)$.  Together these properties justify the terminology ``fractional GJMS operator.''

We adopt the convention that for $\gamma\in\left(0,\frac{n}{2}\right)$, the \emph{fractional $Q$-curvature $Q_{2\gamma}$} is the scalar
\begin{equation}
\label{eqn:q_scattering}
Q_{2\gamma} := \frac{2}{n-2\gamma}P_{2\gamma}(1) .
\end{equation}
In particular, we emphasize that this definition produces a well-defined invariant in the critical case $2\gamma=n$; see~\cite{GrahamZworski2003}.
\section{Smooth metric measure spaces}
\label{sec:smms}

In this article, we show that one can identify fractional GJMS operators as boundary operators associated to weighted GJMS operators in an interior smooth metric measure space.  In order to make sense of this, we must recall some aspects of the (conformal) geometry of smooth metric measure spaces.  Our treatment here will be mostly formal, with definitions being made by passing to formal warped products with fibers of non-integer dimensions.  For a more intrinsic discussion of these topics, we refer the reader to~\cite{Case2011t}.

A \emph{smooth metric measure space} is a four-tuple $(\oX^{n+1},g,v^m\dvol,\mu)$ determined by a Riemannian manifold $(\oX^{n+1},g)$ with boundary $M^n=\partial\oX^{n+1}$, the Riemannian volume element $\dvol$ associated to $g$, a nonnegative function $v\in C^\infty(\oX)$ with $v^{-1}(0)=M$, and constants $m\in\bR\setminus\{1-n\}$ and $\mu\in\bR$.  In the interior $X$ of $\oX$, we define the function $\phi\in C^\infty(X)$ by $v^m=e^{-\phi}$; this is the definition of the symbol $v^m$ in the case $m=\infty$.  For the purposes of this article, we will only be interested in the cases $m\in(1-n,1)$.  Note that smooth metric measure spaces with $m=0$ are Riemannian manifolds.

The only reason for the assumption $m\not=1-n$ is that it allows us to define the weighted Schouten tensor~\eqref{eqn:weighted_schouten_tensor}.  We can allow $m=1-n$ so long as we do not consider the weighted Schouten tensor of $(\oX^{n+1},g,v^m\dvol,\mu)$, a liberty that we will make use of when discussing smooth metric measure spaces as a means to study the fractional GJMS operators $P_{2\gamma}$ in the critical case $2\gamma=n$ for $n$ odd.  For this purpose, note that we can define the weighted conformal Laplacian without using the weighted Schouten tensor; see~\eqref{eqn:weighted_gjms12}.

Formally, one can think of a smooth metric measure space $(X^{n+1},g,v^m\dvol,\mu)$ as the base of the warped product $(X^{n+1}\times F_\mu^m,g\oplus v^2g_F)$ for $(F_\mu^m,g_F)$ the simply-connected $m$-dimensional spaceform with $\Ric(g_F)=\mu g_F$.  This means that one should write down geometric invariants on the warped product (cf.\ \cite{ONeill}) when $m\in\bN$, restrict them to the base, and then extend them to general $m$ by treating $m$ as a formal parameter.  As one example, the \emph{Bakry-\'Emery Ricci tensor} $\Ric_\phi^m$ of $(X^{n+1},g,v^m\dvol,\mu)$ is defined by $\Ric_\phi^m:=\Ric-mv^{-1}\nabla^2v$, and is formally the restriction of the Ricci tensor of the corresponding warped product $(X^{n+1}\times F_\mu^m,g\oplus v^2g_F)$ to horizontal lifts of vector fields on $X$.  As another example, the \emph{weighted Laplacian} $\Delta_\phi$ is the operator $\Delta_\phi:=\Delta-\nabla\phi$, and is formally obtained by applying the Laplacian of the corresponding warped product to the horizontal lift of a smooth function on $X$, and then projecting back to $X$.

We say that $(X^{n+1},g,v^m\dvol_g,\mu)$ and $(X^{n+1},\hat g,\hat v^m\dvol_{\hat g},\mu)$ are \emph{pointwise conformally equivalent} if there is a function $u\in C^\infty(M)$ such that $\hat g=u^{-2}g$ and $\hat v=u^{-1}v$; formally, this is equivalent to the statement that the corresponding warped products are pointwise conformally equivalent.  We define the \emph{weighted Schouten tensor} $P_\phi^m$ and the \emph{weighted Schouten scalar} $J_\phi^m$ of a smooth metric measure space by
\begin{align}
\label{eqn:weighted_schouten_tensor} P_\phi^m & = \frac{1}{m+n-1}\left(\Ric_\phi^m - J_\phi^mg\right) \\
\label{eqn:weighted_schouten_scalar} J_\phi^m & = \frac{1}{2(m+n)}\left(R - 2mv^{-1}\Delta v - m(m-1)v^{-2}\lv\nabla v\rv^2 + m\mu v^{-2}\right) .
\end{align}
Unlike the Riemannian case, the weighted Schouten scalar is not in general the trace of the weighted Schouten tensor.  We shall denote by $Y_\phi^m$ the difference
\[ Y_\phi^m := J_\phi^m - \tr P_\phi^m; \]
note that the Schouten tensor $\widetilde{P}$ of the corresponding warped product satisfies $\widetilde{P}=P\oplus\frac{1}{m}Y_\phi^mv^2g_F$.

Given an integer $k$, the \emph{weighted GJMS operator} $L_{2k,\phi}^m$ of a smooth metric measure space $(X^{n+1},g,v^m\dvol,\mu)$ is the linear conformally covariant differential operator with leading term $(-\Delta_\phi)^k$ which is formally obtained by applying the GJMS operator $L_{2k}$ of the corresponding warped product to the horizontal lift of a function on $X$ and then projecting back to $X$.  In particular, the weighted GJMS operators are conformally covariant; if $(X^{n+1},g,v^m\dvol_g,\mu)$ and $(X^{n+1},\hat g,\hat v^m\dvol_{\hat g},\mu)$ are such that $\hat g=u^{-2}g$ and $\hat v=u^{-1}v$, then
\begin{equation}
\label{eqn:weighted_gjms_covariant}
\widehat{L_{2k,\phi}^m}w = u^{\frac{m+n+1+2k}{2}} L_{2k,\phi}^m\left(u^{-\frac{m+n+1-2k}{2}}w\right)
\end{equation}
for all $w\in C^\infty(X)$.

Local formulae for the weighted GJMS operators of order two and four, and defined without passing to the formal warped product, have been given in~\cite{Case2011t}:
\begin{equation}
\label{eqn:weighted_gjms12}
\begin{split}
L_{2,\phi}^m & = -\Delta_\phi + \frac{m+n-1}{2}J_\phi^m , \\
L_{4,\phi}^m & = \Delta_\phi^2 + \delta_\phi\left(4P_\phi^m - (m+n-1)J_\phi^mg\right)d + \frac{m+n-3}{2}Q_\phi^m
\end{split}
\end{equation}
for
\[ Q_\phi^m := -\Delta_\phi J_\phi^m - 2\left(\lv P_\phi^m\rv^2 + \frac{1}{m}\left(Y_\phi^m\right)^2\right) + \frac{m+n+1}{2}\left(J_\phi^m\right)^2 \]
the \emph{weighted $Q$-curvature}; our sign convention is $\Delta_\phi=\delta_\phi d$ on functions.  Recursive formulae for the GJMS operators of all orders were found by Juhl~\cite{Juhl2013}; see also~\cite{FeffermanGraham2013}.  In principle, these formulae can be extended to the weighted GJMS operators, though at present this has not been carried out in the literature.  This issue will be addressed elsewhere.

In addition to the conformal covariance of the weighted GJMS operator~\eqref{eqn:weighted_gjms_covariant}, we need the following product formulae for the weighted GJMS operators of a special class of smooth metric measure spaces.

\begin{thm}
\label{thm:weighted_gjms_factorization}
Let $(X^{n+1},g_+)$ be a Riemannian manifold with $\Ric(g_+)=-ng_+$.  Then for any $k\in\bN$ and any $m\in(1-n,\infty]$, the weighted GJMS operator $L_{2k,\phi}^m$ of the smooth metric measure space $(X^{n+1},g_+,1^m\dvol,m-1)$ satisfies
\begin{equation}
\label{eqn:weighted_gjms_factorization}
L_{2k,\phi}^m = \prod_{j=1}^k \left(-\Delta_{g_+} - \frac{(n-m+2k-4j+3)(m+n-2k+4j-3)}{4}\right) .
\end{equation}
\end{thm}

The proof of Theorem~\ref{thm:weighted_gjms_factorization} when $k\in\{1,2\}$ is an easy consequence of~\eqref{eqn:weighted_gjms12}.  We prove Theorem~\ref{thm:weighted_gjms_factorization} in Appendix~\ref{app:product} by passing to the formal warped product associated to $(X^{n+1},g_+,1^m\dvol,m-1)$.

In this article, smooth metric measure spaces arise naturally when studying fractional GJMS operators $P_{2\gamma}$ defined on Poincar\'e--Einstein manifolds, with the parameter $\gamma$ determining the value of the dimensional parameter $m$.  The key idea is that there is a natural relationship between the Poisson equation~\eqref{eqn:poisson_equation} and the weighted conformal Laplacian via Theorem~\ref{thm:weighted_gjms_factorization}, and hence the conformal covariance of the latter equation easily yields reformulations of~\eqref{eqn:poisson_equation} and~\eqref{eqn:poisson_asymptotics} in terms of generalized Dirichlet-to-Neumann maps for compactifications of Poincar\'e--Einstein manifolds; see Section~\ref{sec:extension} for details.  The following lemma is useful when performing local computations using this correspondence.

\begin{lem}
\label{lem:pe_smms_formulae}
Let $(X^{n+1},M^n,g_+)$ be a Poincar\'e--Einstein manifold and let $\rho$ be any defining function.  Fix $m\not=1-n$.  The smooth metric measure space $(X^{n+1},g:=\rho^2g_+,\rho^m\dvol_{g},m-1)$ has
\begin{align}
\label{eqn:J_to_grad} J + \rho^{-1}\Delta\rho & = \frac{n+1}{2}\rho^{-2}\left(\lv\nabla\rho\rv^2 - 1 \right) \\
\label{eqn:J_to_weight} J_\phi^m & = J - \frac{m}{n+1}\left(J + \rho^{-1}\Delta\rho\right) \\
\label{eqn:P_to_weight} P_\phi^m & = P .
\end{align}
\end{lem}

\begin{proof}

\eqref{eqn:J_to_grad} follows immediately from the fact that $\rho^{-2}g$ has constant scalar curvature $-n(n+1)$.  By definition,
\[ J_\phi^m = \frac{1}{2(m+n)}\left(2nJ - 2m\rho^{-1}\Delta\rho - m(m-1)\rho^{-2}(\lv\nabla\rho\rv^2-1)\right) . \]
Inserting~\eqref{eqn:J_to_grad} into the above display yields~\eqref{eqn:J_to_weight}.  Since $\rho^{-2}g$ is Einstein, it holds that
\[ P + \rho^{-1}\nabla^2\rho = -\frac{1}{n+1}\left(J+\rho^{-1}\Delta\rho\right)g . \]
Inserting this and~\eqref{eqn:J_to_weight} into the definition
\[ P_\phi^m = \frac{1}{m+n-1}\left( (n-1)P + Jg - m\rho^{-1}\nabla^2\rho - J_\phi^m g \right) \]
of the weighted Schouten tensor yields~\eqref{eqn:P_to_weight}.
\end{proof}
\section{Weighted GJMS operators as boundary operators}
\label{sec:extension}

One of the key observations in this article is that one can naturally regard the fractional GJMS operators as the boundary operators associated to weighted GJMS operators on suitable smooth metric measure spaces.  Phrased differently, one can define the fractional GJMS operators as generalized Dirichlet-to-Neumann maps associated to weighted GJMS operators, providing a natural geometric interpretation of the works of Caffarelli and Silvestre~\cite{CaffarelliSilvestre2007}, Chang and Gonz\'alez~\cite{ChangGonzalez2011}, and R.\ Yang~\cite{Yang2013} on extension problems related to fractional GJMS operators.  The purpose of this section is to make this connection precise.

\subsection{The model case $\gamma\in(0,1)$} By way of motivation, let us begin by reformulating the extension theorem of Chang and Gonz\'alez~\cite{ChangGonzalez2011} for the fractional GJMS operator $P_{2\gamma}$ with $\gamma\in(0,1)$ in the language of smooth metric measure spaces.  Indeed, the following result provides a slight improvement of~\cite[Theorem~4.3 and Theorem~4.7]{ChangGonzalez2011} by allowing one to compute using a much larger class of conformal compactifications.

\begin{thm}
\label{thm:01_case}
Let $(X^{n+1},M^n,g_+)$ be a Poincar\'e--Einstein manifold, fix a representative $h$ of the conformal boundary, and let $r$ be the geodesic defining function associated to $h$.  Let $\gamma\in(0,1)$ and set $m_0=1-2\gamma$.  Let $\rho$ be a defining function for $M$ such that, asymptotically near $M$,
\begin{equation}
\label{eqn:01_case_rho}
\rho = r + \Phi r^{1+2\gamma} + o(r^{1+2\gamma})
\end{equation}
for $\Phi\in C^\infty(M)$, and consider the smooth metric measure space
\begin{equation}
\label{eqn:01_case_smms}
\left(\oX^{n+1},g:=\rho^2g_+,\rho^{m_0}\dvol,m_0-1\right) .
\end{equation}
Given $f\in C^\infty(M)$, the function $U$ is the solution of the boundary value problem
\begin{equation}
\label{eqn:01_case}
\begin{cases}
L_{2,\phi_0}^{m_0} U = 0, & \quad\text{in $X$}, \\
U = f, & \quad\text{on $M$}
\end{cases}
\end{equation}
if and only if the function $u=\rho^{n-s}U$ is the solution of the Poisson problem
\begin{equation}
\label{eqn:01_case_poisson}
\begin{cases}
-\Delta_{g_+}u - s(n-s)u = 0, & \text{in $X$},\\
u = Fr^{n-s} + Gr^s, & F,G\in C^\infty(\oX), \\
F\rv_{r=0} = f
\end{cases}
\end{equation}
for $s=\frac{n}{2}+\gamma$.  Moreover, the solution $U$ of~\eqref{eqn:01_case} is such that
\begin{equation}
\label{eqn:01_extension_defn}
P_{2\gamma}f - \frac{n-2\gamma}{2}d_\gamma\Phi f = \frac{d_\gamma}{2\gamma}\lim_{\rho\to0} \rho^{m_0}\frac{\partial U}{\partial\rho}
\end{equation}
for $d_\gamma$ as in~\eqref{eqn:scattering_definition}.
\end{thm}

\begin{proof}

By the conformal invariance of the weighted conformal Laplacian,
\[ L_{2,\phi_0}^{m_0} U = \rho^{-\frac{m_0+n+3}{2}}\left(L_{2,\phi_0}^{m_0}\right)_+\left(\rho^{\frac{m_0+n-1}{2}}U\right) \]
for $\left(L_{2,\phi_0}^{m_0}\right)_+$ the weighted conformal Laplacian of $(X^{n+1},g_+,1^{m_0}\dvol,m_0-1)$.  By Theorem~\ref{thm:weighted_gjms_factorization} and the choice of $m_0$,
\[ \left(L_{2,\phi_0}^{m_0}\right)_+ = -\Delta_{g_+} - \frac{(n-2\gamma)(n+2\gamma)}{4} . \]
In particular, $U$ is a solution to~\eqref{eqn:01_case} if and only if $u=\rho^{n-s}U$ is a solution to~\eqref{eqn:01_case_poisson}.  It thus follows from~\eqref{eqn:poisson_asymptotics}, \eqref{eqn:F_expansion}, and~\eqref{eqn:01_case_rho} that, for such a solution $U$,
\begin{equation}
\label{eqn:01_expansion}
U = f + \left(d_\gamma^{-1}P_{2\gamma}f - \frac{n-2\gamma}{2}\Phi f\right)\rho^{2\gamma} + o(\rho^{2\gamma}),
\end{equation}
from which the conclusion immediately follows.
\end{proof}

Note that Theorem~\ref{thm:01_case} puts very mild assumptions on the choice of defining function, and in particular it applies to all defining functions which are smooth up to the boundary $M$ (in which case $\Phi=0$).  This allows us to be very flexible in making our choice of defining function.  An especially important defining function is constructed in Lemma~\ref{lem:01_rhostar}, where the term $\Phi$ is a constant multiple of $Q_{2\gamma}$.

As an immediate corollary of Theorem~\ref{thm:01_case} we obtain the following energy identity relating a fractional GJMS operator of order $2\gamma\in(0,2)$ on the boundary and the weighted conformal Laplacian of the corresponding smooth metric measure space in the interior.

\begin{cor}
\label{cor:01_inequality}
Let $(X^{n+1},M^n,g_+)$ and $(\oX^{n+1},g,\rho^{m_0}\dvol,m_0-1)$ be as in Theorem~\ref{thm:01_case}.  Given $f\in C^\infty(M)$, let $U$ be the solution of~\eqref{eqn:01_case}.  Then
\begin{multline}
\label{eqn:01_energy_inequality}
-\frac{2\gamma}{d_\gamma}\left[\int_M f\,P_{2\gamma}f\,\dvol_h - \frac{n-2\gamma}{2}d_\gamma\int_M \Phi f^2\,\dvol_h\right] \\ = \int_X \left(\lv\nabla U\rv^2 + \frac{m_0+n-1}{2}J_{\phi_0}^{m_0} U^2\right)\rho^{m_0}\dvol_g .
\end{multline}
\end{cor}

Note that if $\gamma\in(0,1)$, then $d_\gamma<0$; see~\eqref{eqn:scattering_definition}.

\subsection{The cases $\gamma\in\left(1,\frac{n}{2}\right)\setminus\bN$} Let us now consider the analogues of Theorem~\ref{thm:01_case} and Corollary~\ref{cor:01_inequality} for larger values of $\gamma$.  Using the equivalence of~\eqref{eqn:01_case} with the Poisson problem~\eqref{eqn:01_case_poisson}, Chang and Gonz\'alez showed~\cite{ChangGonzalez2011} that one can identify the fractional GJMS operator $P_{2\gamma}$ for all $\gamma\in(0,\frac{n}{2})\setminus\bN$ by solving~\eqref{eqn:01_case} and taking sufficiently many normal derivatives on the boundary.  However, it is difficult to use this observation to study the analytic properties of $P_{2\gamma}$; in particular, the energy identity~\eqref{eqn:01_energy_inequality} is no longer valid.

R.\ Yang showed~\cite{Yang2013} how to realize fractional powers of the Laplacian in the flat Euclidean space $\bR^n$ via extensions in such a way as to retain an energy identity relating $(-\Delta)^\gamma$ to an operator defined in the interior of $\bR_+^{n+1}$.  His key insight is that one should understand $(-\Delta)^\gamma$ in terms of an extension using an operator of order $2\lfloor\gamma\rfloor+2$.

\begin{thm}[{R.\ Yang~\cite[Theorem~4.1]{Yang2013}}]
\label{thm:yang}
Let $0<\gamma\not\in\bN$ and set $k=\lfloor\gamma\rfloor$ and $m_k=2k+1-2\gamma\in(-1,1)$.  Define $\Delta_{m_k}:=\Delta+m_ky^{-1}\partial_y$.  Suppose that $U\in W^{k+1,2}(\bR_+^{n+1},y^{m_k}\dvol)$ satisfies
\begin{equation}
\label{eqn:yang_case}
\begin{cases}
 \Delta_{m_k}^{k+1}U = 0, & \text{in $\bR_+^{n+1}$} \\
 \frac{\partial^{2\ell}U}{\partial y^{2\ell}}=(\Delta^\ell f)\prod_{j=0}^{\ell-1}\frac{1}{2\gamma-4j}, & \text{on $M$, where $\ell\in\{0,1,\dotsc,\lfloor\frac{k}{2}\rfloor\}$}, \\
 \lim_{y\to0}y^{m_k}\frac{\partial^{2\ell+1}U}{\partial y^{2\ell+1}} = 0, & \text{where $\ell\in\{0,1,\dotsc\lfloor\frac{k-1}{2}\rfloor\}$}, \\
\end{cases}
\end{equation}
for some $f\in H^\gamma(\bR^n)$, where our convention is that the empty product is equal to one.  Then
\begin{equation}
\label{eqn:yang_extension_defn}
(-\Delta)^\gamma f = C_{\gamma}\lim_{y\to0} y^{m_k}\frac{\partial}{\partial y}\Delta_{m_k}^kU(x,y)
\end{equation}
for $C_{\gamma}$ an explicit constant depending only on $\gamma$.
\end{thm}

We highlight two features of Theorem~\ref{thm:yang}.  First, in terms of the smooth metric measure space $(\bR_+^{n+1},dx^2\oplus dy^2,y^{m_k}\dvol,m_k-1)$, the operator $(-\Delta_{m_k})^{k+1}$ is exactly the weighted GJMS operator of order $2k+2$.  Second, the solution to~\eqref{eqn:yang_case} is unique and is shown in~\cite{Yang2013} to correspond to the solution of the Poisson equation~\eqref{eqn:poisson_equation} with $F\rv_{\bR^n}=f$; in particular, this identifies the terms $(\Delta^\ell f)\prod_{j=0}^{\ell-1}\frac{1}{2\gamma-4j}$ as constant multiples of the coefficients $f_{(2\ell)}$ in the expansion~\eqref{eqn:F_expansion}.

From these observations, it is natural to expect that the general curved analogue of Theorem~\ref{thm:yang} can be obtained by replacing occurrences of $\Delta_{m_k}^{k+1}$ with the appropriate weighted GJMS operators.  The remainder of this section is devoted to showing that this is the case.  In order to motivate the arguments and to isolate the key result needed for our study in Section~\ref{sec:positivity} of the positivity of the fractional GJMS operators $P_{2\gamma}$ with $\gamma\in(1,2)$, we begin by considering the case $\gamma\in(1,2)$ and then consider the general case.  The following result is a generalization of Theorem~\ref{thm:intro_12_case} from the introduction, where now we allow for a rather general choice of defining function.

\begin{thm}
\label{thm:12_case}
Let $(X^{n+1},M^n,g_+)$ be a Poincar\'e--Einstein manifold, fix a representative $h$ of the conformal boundary, and let $r$ be the geodesic defining function associated to $h$.  Let $\gamma\in(1,2)$ and set $m_1=3-2\gamma$.  Let $\rho$ be a defining function for $M$ such that, asymptotically near $M$,
\begin{equation}
\label{eqn:12_case_rho}
\rho = r + \rho_{(2)}r^3 + \Phi r^{1+2\gamma} + o(r^{1+2\gamma})
\end{equation}
for functions $\rho_{(2)},\Phi\in C^\infty(M)$, and consider the smooth metric measure space
\begin{equation}
\label{eqn:12_case_smms}
\left(\oX^{n+1},g:=\rho^2g_+,\rho^{m_1}\dvol,m_1-1\right) .
\end{equation}
Given $f\in C^\infty(M)$, the function $U$ is the solution of the boundary value problem
\begin{equation}
\label{eqn:12_case}
\begin{cases}
L_{4,\phi_1}^{m_1} U = 0, & \quad\text{in $X$}, \\
U = f, & \quad\text{on $M$}, \\
\lim_{\rho\to0}\rho^{m_1}\frac{\partial U}{\partial\rho} = 0
\end{cases}
\end{equation}
if and only if the function $u=\rho^{n-s}U$ is the solution of the Poisson problem
\begin{equation}
\label{eqn:12_case_poisson}
\begin{cases}
-\Delta_{g_+}u - s(n-s)u = 0, & \text{in $X$},\\
u = Fr^{n-s} + Gr^s, & F,G\in C^\infty(\oX), \\
F\rv_{r=0} = f
\end{cases}
\end{equation}
for $s=\frac{n}{2}+\gamma$.  Moreover, the solution $U$ of~\eqref{eqn:12_case} is such that
\begin{equation}
\label{eqn:12_extension_defn}
P_{2\gamma}f - \frac{n-2\gamma}{2}d_\gamma\Phi f = \frac{d_\gamma}{8\gamma(\gamma-1)}\lim_{\rho\to0} \rho^{m_1}\frac{\partial}{\partial\rho}\Delta_{\phi_1}U
\end{equation}
for $d_\gamma$ as in~\eqref{eqn:scattering_definition}.
\end{thm}

\begin{proof}

By the conformal invariance of the weighted Paneitz operator,
\[ L_{4,\phi_1}^{m_1} U = \rho^{-\frac{m_1+n+5}{2}}\left(L_{4,\phi_1}^{m_1}\right)_+\left(\rho^{\frac{m_1+n-3}{2}}U\right) \]
for $\left(L_{4,\phi_1}^{m_1}\right)_+$ the weighted Paneitz operator of $(X^{n+1},g_+,1^{m_1}\dvol,m_1-1)$.  By Theorem~\ref{thm:weighted_gjms_factorization} and the choice of $m_1$,
\[ \left(L_{4,\phi_1}^{m_1}\right)_+ = \left(-\Delta_{g_+} - \frac{(n-2\gamma+4)(n+2\gamma-4)}{4}\right)\left(-\Delta{g_+} - \frac{(n-2\gamma)(n+2\gamma)}{4}\right) . \]
In particular, $L_{4,\phi_1}^{m_1}U=0$ if and only if
\[ \left(-\Delta_{g_+}-(s-2)(n-s+2)\right)\left(-\Delta_{g_+}-s(n-s)\right)u = 0 \]
for $s=\frac{n}{2}+\gamma$ and $U=\rho^{s-n}u$.  From~\eqref{eqn:poisson_asymptotics}, \eqref{eqn:F_expansion}, and~\eqref{eqn:12_case_rho} we know that the solution $\tilde u$ to~\eqref{eqn:12_case_poisson} satisfies
\begin{equation}
\label{eqn:12_expansion}
\tilde U = f + \left(f_{(2)}-\frac{n-2\gamma}{2}\rho_{(2)}\right)\rho^2 + \left(d_\gamma^{-1}P_{2\gamma}f - \frac{n-2\gamma}{2}\Phi f\right)\rho^{2\gamma} + o(\rho^{2\gamma})
\end{equation}
for $\tilde U:=\rho^{s-n}\tilde u$; in particular, $\tilde U$ satisfies $L_{4,\phi_1}^{m_1}\tilde U=0$ and $\rho^{m_1}\partial_\rho\tilde U\to0$ as $\rho\to0$.  By uniqueness of solutions to~\eqref{eqn:12_case}, it follows that $U=\tilde U$; i.e.\ we have the claimed one-to-one correspondence between the solutions of~\eqref{eqn:12_case} and the solutions of~\eqref{eqn:12_case_poisson}.

Next, from~\eqref{eqn:12_expansion} it follows that
\begin{equation}
\label{eqn:12_extension_predefn}
P_{2\gamma}f - \frac{n-2\gamma}{2}\Phi d_\gamma f = \frac{d_\gamma}{4\gamma(\gamma-1)}\lim_{\rho\to0}\rho^{m_1}\frac{\partial}{\partial\rho}\left(\rho^{-1}\frac{\partial U}{\partial\rho}\right)
\end{equation}
for $U$ a solution to~\eqref{eqn:12_case}.  On the other hand, the discussion of the previous paragraph shows that $U$ satisfies $L_{2,\phi_0}^{m_0} U=0$, where $L_{2,\phi_0}^{m_0}$ is the weighted conformal Laplacian of the smooth metric measure space~\eqref{eqn:01_case_smms}.  Therefore
\begin{equation}
\label{eqn:12_case_Delta-to-L}
-\Delta_{\phi_1}U = -2\rho^{-1}\frac{\partial U}{\partial\rho} - \frac{m_0+n-1}{2}J_{\phi_0}^{m_0}U .
\end{equation}

We claim that $\partial_\rho J_{\phi_0}^{m_0}=O(\rho)$.  Indeed, from~\eqref{eqn:12_case_rho} we readily compute that
\[ \rho^{-2}\left(\lv\nabla\rho\rv^2 - 1\right) = 4\rho_{(2)} + 4\gamma\Phi r^{2\gamma-2} + O(r^2) , \]
and hence, by Lemma~\ref{lem:pe_smms_formulae},
\begin{equation}
\label{eqn:12_case_JrhoDelta}
J + \rho^{-1}\Delta\rho = 2(n+1)\left(\rho_{(2)} + \gamma\Phi r^{2\gamma-2}\right) + O(r^2) .
\end{equation}
By writing $g = \left(\frac{\rho}{r}\right)^2 (dr^2\oplus g_r)$ and recalling that the asymptotic expansion of $h_r$ near $M$ is $h_r=h+h_{(2)}r^2+o(r^2)$, we readily compute that
\[ \rho^{-1}\Delta\rho = 2(n+2)\rho_{(2)} + 2\tr_h h_{(2)} + 2\gamma(n+2\gamma)\Phi r^{2\gamma-2} + O(r^2) \]
and hence, by~\eqref{eqn:12_case_JrhoDelta},
\begin{equation}
\label{eqn:12_case_J}
J = -2\tr h_{(2)} - 2\rho_{(2)} + 2m_0\gamma\Phi r^{2\gamma-2} + O(r^2) .
\end{equation}
Lemma~\ref{lem:pe_smms_formulae} then gives
\[ J_{\phi_0}^{m_0} = -2\tr h_{(2)} - 4(1-\gamma)\rho_{(2)} + O(r^2) , \]
yielding our claim.

Finally, \eqref{eqn:12_case_Delta-to-L} and the above asymptotic behavior of $\partial_\rho J_{\phi_0}^{m_0}$ shows that
\[ -\lim_{\rho\to0} \rho^{m_1}\frac{\partial}{\partial\rho}\Delta_{\phi_1}U = -2\lim_{\rho\to0}\rho^{m_1}\frac{\partial}{\partial\rho}\left(\rho^{-1}\frac{\partial U}{\partial\rho}\right) . \]
Inserting this into~\eqref{eqn:12_extension_predefn} yields the desired result.
\end{proof}

As an immediate corollary, we see that we may identify the energy of the fractional GJMS operators $P_{2\gamma}$ with $\gamma\in(1,2)$ and the energy in the interior of the weighted Paneitz operator of the corresponding smooth metric measure space.

\begin{cor}
\label{cor:12_inequality}
Let $(X^{n+1},M^n,g_+)$ and $(\oX^{n+1},g,\rho^{m_1}\dvol,m_1-1)$ be as in Theorem~\ref{thm:12_case}.  Let $f\in C^\infty(M)$ and let $U$ be the solution of~\eqref{eqn:12_case}.  Then
\begin{multline}
\label{eqn:12_energy_inequality}
\frac{8\gamma(\gamma-1)}{d_\gamma}\left[\int_M f\,P_{2\gamma}f\,\dvol_h - \frac{n-2\gamma}{2}d_\gamma\int_M \Phi f^2\,\dvol_h\right] \\
= \int_X \left((\Delta_{\phi_1} U)^2 - (4P-(m_1+n-1)J_{\phi_1}^{m_1}g)(\nabla U,\nabla U) + \frac{m_1+n-3}{2}Q_{\phi_1}^{m_1}U^2\right)\rho^{m_1}\dvol_g .
\end{multline}
\end{cor}

Note that if $\gamma\in(1,2)$, then $d_\gamma>0$; see~\eqref{eqn:scattering_definition}.

By following the same ideas as in the proof of Theorem~\ref{thm:12_case}, we also have the following curved analogue of Theorem~\ref{thm:yang}.

\begin{thm}
\label{thm:general_case}
Let $(X^{n+1},M^n,g_+)$ be a Poincar\'e--Einstein manifold, fix a representative $h$ for the conformal boundary, and let $r$ be the geodesic defining function associated to $h$.  Given $\gamma\in(0,\frac{n}{2})\setminus\bN$, set $k=\lfloor\gamma\rfloor$ and $m_k=2k+1-2\gamma$, and consider the smooth metric measure space
\[ \left(\oX^{n+1},g:=r^2g_+,r^{m_k}\dvol,m_k-1\right) . \]
Given $f\in C^\infty(M)$, the function $U$ is the solution of the boundary value problem
\begin{equation}
\label{eqn:general_case}
\begin{cases}
L_{2k+2,\phi_k}^{m_k} U = 0, & \quad\text{in $X$}, \\
\frac{\partial^{2\ell}U}{\partial r^{2\ell}}=(2\ell)!f_{(2\ell)}, & \quad\text{on $M$, where $\ell\in\{0,1,\dotsc,\lfloor\frac{k}{2}\rfloor\}$}, \\
\lim_{r\to0}r^{m_k}\frac{\partial^{2\ell+1}U}{\partial r^{2\ell+1}} = 0, & \quad\text{where $\ell\in\{0,1,\dotsc\lfloor\frac{k-1}{2}\rfloor\}$}, \\
\end{cases}
\end{equation}
where the functions $f_{(2\ell)}$ are as in~\eqref{eqn:F_expansion}, if and only if the function $u=r^{n-s}U$ is the solution of the Poisson problem
\begin{equation}
\label{eqn:general_case_poisson}
\begin{cases}
-\Delta_{g_+}u - s(n-s)u = 0, & \text{in $X$},\\
u = Fr^{n-s} + Gr^s, & F,G\in C^\infty(\oX), \\
F\rv_{r=0} = f .
\end{cases}
\end{equation}
Moreover, the solution $U$ of~\eqref{eqn:general_case} is such that
\begin{equation}
\label{eqn:general_extension_defn}
P_{2\gamma}f = \frac{\Gamma(\gamma-k)}{\Gamma(\gamma+1)}\frac{(-1)^kd_\gamma}{2^{2k+1}k!}\lim_{r\to0} r^{m_k}\frac{\partial}{\partial r}L_{2k,\phi_k}^{m_k}U
\end{equation}
for $d_\gamma$ as in~\eqref{eqn:scattering_definition}.
\end{thm}

\begin{remark}
Note that we have stated Theorem~\ref{thm:general_case} in terms of the geodesic defining function only.  This is so that we can use the expansion~\eqref{eqn:F_expansion} to more easily formulate the boundary conditions in~\eqref{eqn:general_case} necessary to conclude that solutions of~\eqref{eqn:general_case} are solutions to the Poisson equation~\eqref{eqn:poisson_equation}.  If one takes instead a defining function of the form
\[ \rho = r\left(1 + \rho_{(2)}r^2 + \rho_{(4)}r^4 + \dotso + \rho_{(2k)}r^{2k} + \Phi r^{2\gamma} + o(r^{2\gamma})\right), \]
such as arises when considering smooth defining functions (in which case $\Phi=0$) or the analogues of the special defining function $y$ constructed in Section~\ref{sec:adapted}, one must replace the boundary conditions in~\eqref{eqn:general_case} involving the functions $f_{(2\ell)}$ with terms also taking into account the terms $\rho_{(2\ell)}$ and one must also include in~\eqref{eqn:general_extension_defn} additional terms involving $\Phi$ (cf.\ Theorem~\ref{thm:01_case} and Theorem~\ref{thm:12_case}).
\end{remark}

\begin{proof}

By the conformal invariance of the weighted Paneitz operator,
\[ L_{2k+2,\phi_k}^{m_k}U = r^{-\frac{m_k+n+2k+3}{2}}\left(L_{2k+2,\phi_k}^{m_k}\right)_+\left(r^{\frac{m_k+n-2k-1}{2}}U\right) \]
for $\left(L_{2k+2,\phi_k}^{m_k}\right)_+$ the weighted Paneitz operator of $(X^{n+1},g_+,1^{m_k}\dvol,m_k-1)$.  By Theorem~\ref{thm:weighted_gjms_factorization} and the choice of $m_k$,
\begin{multline*}
\left(L_{2k+2,\phi}^{m_k}\right)_+ = \left[\prod_{j=2}^{k+1} \left(-\Delta_{g_+} - \frac{(n-4j+4+2\gamma)(n+4j-4-2\gamma)}{4}\right)\right]\\\circ\left[-\Delta_{g_+} - \frac{(n-2\gamma)(n+2\gamma)}{4}\right] .
\end{multline*}
In particular, if $-\Delta_{g_+}u-s(n-s)u=0$ for $s=\frac{n}{2}+\gamma$, then $\tilde U:=r^{s-n}u$ satisfies $L_{2k+2,\phi_k}^{m_k}\tilde U=0$.  Moreover, from~\cite{GrahamZworski2003,MazzeoMelrose1987}, we know that if $\tilde U$ is chosen so that $U\rv_M=f$, then
\begin{equation}
\label{eqn:general_expansion}
\tilde U = f + f_{(2)}r^2 + \dotsb + f_{(2k)}r^{2k} + d_\gamma^{-1}r^{2\gamma}P_{2\gamma}f + O(r^{2k+2}) .
\end{equation}
Thus $\tilde U$ satisfies~\eqref{eqn:general_case}, and hence, by uniqueness of solutions to~\eqref{eqn:general_case}, $\tilde U=U$.

Now, from~\eqref{eqn:general_expansion} (cf.\ \cite{ChangGonzalez2011}) it follows that
\[ P_{2\gamma}f = 2^{-k-1}\frac{d_\gamma\Gamma(\gamma-k)}{\Gamma(\gamma+1)}\lim_{r\to0}r^{m_k}\frac{\partial}{\partial r}\left(r^{-1}\frac{\partial}{\partial r}\right)^kU . \]
That this is equivalent to the desired equation~\eqref{eqn:general_extension_defn} follows from Proposition~\ref{prop:induction} below.
\end{proof}

\begin{prop}
\label{prop:induction}
Let $(X^{n+1},M^n,g_+)$ be a Poincar\'e--Einstein manifold.  Fix a defining function $\rho$ such that
\begin{equation}
\label{eqn:defining_function_induction}
\rho = r + \rho_{(2)}r^3 + \rho_{(4)}r^5 + \dotsb
\end{equation}
for $r$ the geodesic defining function associated to $\rho^2g_+\rv_{TM}$; in other words, suppose the Taylor series expansion of $\rho/r$ near $M$ is even in $r$.  Let $\gamma\in\left(0,\frac{n}{2}\right]\setminus\bN$ and set $m_0=1-2\gamma$.  Suppose that $L_{2,\phi_0}^{m_0}U=0$.  Then for any $k\in\bN$ it holds that
\begin{equation}
\label{eqn:induction}
L_{2k,\phi_k}^{m_k}U \cong (-2)^k k!\left(\rho^{-1}\frac{\partial}{\partial\rho}\right)^kU,
\end{equation}
where $L_{2k,\phi_k}^{m_k}$ is the weighted GJMS operator of order $2k$ of the smooth metric measure space
\[ \left( X^{n+1}, \rho^2g_+, \rho^{m_k}\dvol, m_k-1 \right), \qquad m_k:=m+2k \]
and~\eqref{eqn:induction} means that, in terms of the power series expansions in $\rho$ near the boundary $M$, the coefficients of $\rho^{2\gamma-2k}$ of both sides agree.
\end{prop}

The proof of Proposition~\ref{prop:induction} depends on two simple observations relating the weighted GJMS operators of the same order of the smooth metric measure spaces $(X^{n+1},g,\rho^\ell\dvol,\ell-1)$ and $(X^{n+1},g,\rho^m\dvol,m-1)$.

\begin{lem}
\label{lem:use_factorization}
Let $(X^{n+1},M^n,g_+)$ be a Poincar\'e--Einstein manifold and fix a smooth defining function $\rho$.  For each $m\in(1-n,\infty)$, consider the smooth metric measure space $(X^{n+1},g:=\rho^2g_+,\rho^m\dvol,m-1)$ with its associated weighted GJMS operators $L_{2k,\phi}^m$.  Then
\begin{equation}
\label{eqn:use_factorization}
L_{2k,\phi}^m = \prod_{j=1}^k L_{2,\phi}^{m-2k+4j-2} .
\end{equation}
In particular,
\begin{equation}
\label{eqn:use_factorization_1k-1}
L_{2k,\phi}^m = L_{2,\phi}^{m+2k-2}\circ L_{2k-2,\phi}^{m-2} .
\end{equation}
\end{lem}

\begin{proof}

This is an immediate consequence of Theorem~\ref{thm:weighted_gjms_factorization}.
\end{proof}

In the next lemma we establish the equivalence between the highest order terms of operators on $X$, where we measure ``order'' in terms of Taylor series expansions in a defining function at the boundary of a Poincar\'e--Einstein manifold; for example, both the weighted conformal Laplacian $L_{2,\phi}^m$ and the operator $\rho^{-1}\partial_\rho$ are second-order in this sense.  Given two differential operators $A_k$ and $B_k$ of order $2k$ in this sense, we will write $A_k\cong B_k$ if the leading order terms of $A_k$ and $B_k$ agree; this implies that~\eqref{eqn:induction} holds.

\begin{lem}
\label{lem:sl2}
Let $(X^{n+1},M^n,g_+)$ and $(X^{n+1},g,\rho^m\dvol,m-1)$ be as in Lemma~\ref{lem:use_factorization}.  Suppose additionally that $\rho$ is of the form~\eqref{eqn:defining_function_induction}.  Given any $\ell$ such that $m-\ell>1-n$, it holds that
\begin{align}
\label{eqn:sl2_id} L_{2,\phi}^m - L_{2,\phi}^{m-\ell} & \cong -\ell\rho^{-1}\partial_\rho \\
\label{eqn:sl2_commute} L_{2,\phi}^m\circ\left(\rho^{-1}\partial_\rho\right) & \cong \rho^{-1}\partial_\rho \circ L_{2,\phi}^{m-2} .
\end{align}
\end{lem}

\begin{proof}

Since $\rho$ is of the form~\eqref{eqn:defining_function_induction}, we readily check that the asymptotic expansions of $R$, $\rho^{-1}\Delta\rho$, and $\rho^{-2}\left(\lv\nabla\rho\rv^2-1\right)$ near $M$ all contain only even powers of $r$.  From the definition of the weighted conformal Laplacian, it then follows that~\eqref{eqn:sl2_id} holds.  Using~\eqref{eqn:defining_function_induction} again, we compute that
\begin{align*}
L_{2,\phi}^m\left(\rho^{-1}\partial_\rho(U)\right) & \cong -\left(\Delta + m\rho^{-1}\partial_\rho\right)\left(\rho^{-1}\partial_\rho U\right) \\
& \cong -\rho^{-1}\partial_\rho\left(\Delta U\right) - (m-2)\left(\rho^{-1}\partial_\rho\right)^2U \\
& \cong \rho^{-1}\partial_\rho\left(L_{2,\phi}^{m-2}U\right) . \qedhere
\end{align*}
\end{proof}

\begin{proof}[Proof of Proposition~\ref{prop:induction}]

The proof is by induction.  From~\eqref{eqn:sl2_id} we see that~\eqref{eqn:induction} holds when $k=1$.  Suppose now that it holds for a given value of $k$.  By~\eqref{eqn:use_factorization_1k-1} we may write
\[ L_{2k+2,\phi_{k+1}}^{m_{k+1}}U = L_{2,\phi}^{m+4k+2}\left(L_{2k,\phi_k}^{m_k}U\right) . \]
By applying the inductive hypothesis and then~\eqref{eqn:sl2_commute} to the above display we find that
\begin{align*}
L_{2k+2,\phi_{k+1}}^{m_{k+1}}U & \cong (-2)^kk! L_{2,\phi}^{m+4k+2}\left((\rho^{-1}\partial_\rho)^kU\right) \\
& \cong (-2)^kk!\left(\rho^{-1}\partial_\rho\right)^k L_{2,\phi}^{m+2k+2}U .
\end{align*}
From~\eqref{eqn:sl2_id} and the assumption $L_{2,\phi_0}^{m_0}U=0$ it thus follows that
\[ L_{2k+2,\phi_{k+1}}^{m_{k+1}}U \cong (-2)^{k+1}(k+1)!\left(\rho^{-1}\partial_\rho\right)^{k+1}U, \]
as desired.
\end{proof}
\section{A renormalized energy identity}
\label{sec:energy}

One can attribute the failure of Corollary~\ref{cor:01_inequality} when $\gamma>1$ to the fact that the integral on the right hand side of~\eqref{eqn:01_energy_inequality} is infinite in this case while the energy on the left hand side is finite.  Nevertheless, R.\ Yang~\cite{Yang2013} observed that both sides of~\eqref{eqn:01_energy_inequality} are still related after renormalizing the divergent integral.  In the simplest case, namely $\gamma=3/2$, this was written down explicitly in~\cite[Section~2.2]{Yang2013} as follows.

\begin{thm}
\label{thm:yang_renormalized_energy}
Let $f\in H^{3/2}(\bR^n)$ and let $U\in W^{2,2}(\bR_+^{n+1})$ be the solution of the extension problem~\eqref{eqn:yang_case} with $\gamma=3/2$.  Then
\begin{equation}
\label{eqn:yang_renormalized_energy}
\frac{1}{2}\int_{\bR_+^{n+1}} \left(\Delta U\right)^2 = \lim_{\varepsilon\to0} \left(\frac{1}{\varepsilon}\int_{\bR^n}\lv\nabla f\rv^2 - \int_{\{y>\varepsilon\}}\lv\nabla U\rv^2\,y^{-2}\right) .
\end{equation}
\end{thm}

From Theorem~\ref{thm:yang} it follows that the left hand side of~\eqref{eqn:yang_renormalized_energy} can be identified, up to a multiplicative constant, with the energy $\left((-\Delta)^{3/2}f,f\right)$.  On the other hand, the first integral on the right hand side of~\eqref{eqn:yang_renormalized_energy} is, up to a multiplicative constant, equal to $\int_{\bR^n}f\,f_{(2)}$ for $f_{(2)}$ defined in terms of $U$ via~\eqref{eqn:poisson_asymptotics} and~\eqref{eqn:F_expansion}.  Thus we may interpret Theorem~\ref{thm:yang_renormalized_energy} as stating that the energy of $(-\Delta)^{3/2}$ is the finite part of the energy of the scattering operator $-\Delta-\frac{n^2-(3/2)^2}{4}$, and the infinite part is determined by $f_{(2)}$ in the expansion~\eqref{eqn:F_expansion}.

The following theorem states that this interpretation persists in the curved setting when considering the weighted GJMS operators $P_{2\gamma}$ with $\gamma\in(1,2)$.  For simplicity we have opted to state our result using only geodesic defining functions; it is straightforward to use the asymptotics expansions given in the proof of Theorem~\ref{thm:12_case} to consider more general defining functions.  Also, we have opted to state our result in terms of the energy of $P_{2\gamma}$.  Using Corollary~\ref{cor:12_inequality}, one can instead write this result in terms of the energy of the weighted Paneitz operator in the interior, and thereby realize Theorem~\ref{thm:12_renormalized_energy} as the curved analogue of Theorem~\ref{thm:yang_renormalized_energy}.

\begin{thm}
\label{thm:12_renormalized_energy}
Let $(X^{n+1},M^n,g_+)$ be a Poincar\'e--Einstein manifold, fix a representative $h$ of the conformal boundary, and let $r$ be the geodesic defining function associated to $h$.  Let $\gamma\in(1,2)$, set $m_1=3-2\gamma$, and consider the smooth metric measure space~\eqref{eqn:12_case_smms} determined by $r$.  Given $f\in C^\infty(M)$, let $U$ be the solution of~\eqref{eqn:12_case}.  Then
\begin{equation}
\label{eqn:12_renormalized_energy}
\begin{split}
\frac{4\gamma(\gamma-1)}{d_\gamma}\int_M f\,P_{2\gamma}f & = \lim_{\varepsilon\to0} \bigg[ \varepsilon^{2-2\gamma}\int_M\left(\lv\nabla f\rv^2 + \frac{n-2\gamma}{2}J_hf^2\right)\dvol_h \\
& \qquad - 2(\gamma-1)\int_{r>\varepsilon}\left(\lv\nabla U\rv^2 + \frac{m_0+n-1}{2}J_{\phi_0}^{m_0}U^2\right)r^{m_0}\dvol_g \bigg] .
\end{split}
\end{equation}
where the second summand on the right hand side is computed in terms of the smooth metric measure space~\eqref{eqn:01_case_smms} determined by $r$.
\end{thm}

\begin{proof}

Let $U$ be the solution of~\eqref{eqn:12_case}, so also $L_{2,\phi_0}^{m_0}U=0$.  We compute that
\begin{equation}
\label{eqn:formal_computation}
\begin{split}
&\int_{r>\varepsilon} \left(\lv\nabla U\rv^2 + \frac{m_0+n-1}{2}J_{\phi_0}^{m_0}U^2\right)r^{m_0}\dvol \\
& = \int_{r>\varepsilon} U\,L_{2,\phi_0}^{m_0}U\,r^{m_0}\dvol - \varepsilon^{1-2\gamma}\int_{r=\varepsilon} U\frac{\partial U}{\partial r}\,\dvol_{h_\varepsilon} \\
& = - \varepsilon^{2-2\gamma}\int_{r=\varepsilon} U\left(2f_{(2)} + \frac{2\gamma}{d_\gamma}\varepsilon^{2\gamma-2}P_{2\gamma}f + O(\varepsilon^2)\right)\dvol_{h_\varepsilon} \\
& = - 2\varepsilon^{2-2\gamma}\int_{r=\varepsilon} f\,f_{(2)}\,\dvol_h - \frac{2\gamma}{d_\gamma}\int_{r=\varepsilon} f\,P_{2\gamma}f\,\dvol_h + O(\varepsilon^{4-2\gamma}),
\end{split}
\end{equation}
where the second equality follows from~\eqref{eqn:12_expansion}.  Since
\begin{equation}
\label{eqn:f2}
f_{(2)} = -\frac{1}{4(\gamma-1)}\left(-\Delta_{h} + \frac{n-2\gamma}{2}J_{h}\right)f
\end{equation}
(see~\cite{GrahamZworski2003} or~\cite[Equation~(22)]{GuillarmouQing2010}), integrating by parts on $M$ and taking the limit $\varepsilon\to0$ in~\eqref{eqn:formal_computation} yields~\eqref{eqn:12_renormalized_energy}.
\end{proof}

It is clear from the derivation of~\eqref{eqn:formal_computation} in the above proof that a similar statement for the renormalized energy of the scattering operator exists for all $\gamma$.  More precisely, using the asymptotics used in the proof of Theorem~\ref{thm:general_case}, one can write
\begin{align*}
& \int_{r>\varepsilon}\left(\lv\nabla U\rv^2 + \frac{m_0+n-1}{2}J_{\phi_0}^{m_0}U^2\right)r^{m_0}\dvol \\
& = I_{(2)}\varepsilon^{2-2\gamma} + I_{(4)}\varepsilon^{4-2\gamma} + \dotsb + I_{(2k)}\varepsilon^{2k-2\gamma} + \frac{2\gamma}{d_\gamma}\int_{r=\varepsilon} f\,P_{2\gamma}f\,\dvol_h + O\left(\varepsilon^{2(k+1-\gamma)}\right) ,
\end{align*}
where $k=\lfloor\gamma\rfloor$ and $I_{(2j)}$ are boundary integrals of local invariants on the boundary constructed from $f_{(2\ell)}$ and $v_{(2\ell-2)}$ for $\ell\in\{0,\dotsc,j\}$ and $j\in\{1,\dotsc,k\}$ and $v_{(2\ell)}$ the renormalized volume coefficients~\cite{Graham2000}.
\section{The adapted smooth metric measure space}
\label{sec:adapted}

While our extension theorems from Section~\ref{sec:extension} provide the means to study fractional GJMS operators $P_{2\gamma}$ by working in the interior of smooth metric measure spaces, they do not immediately suggest how to use assumptions on the fractional $Q$-curvature $Q_{2\gamma}$ to control $P_{2\gamma}$, such as in Theorem~\ref{thm:positive}.  In order to see the influence of the fractional $Q$-curvature, we introduce in this section the adapted smooth metric measure space --- which is exactly the conformal compactification obtained using the defining function $\rho^\ast$ constructed in~\cite[Lemma~4.5]{ChangGonzalez2011} --- and study some of its basic properties.

The adapted smooth metric measure space is a (non-smooth) compactification of a Poincar\'e--Einstein manifold associated to a choice of $\gamma\in(0,\frac{n}{2})\setminus\bN$ and a choice of representative $h$ of the conformal boundary which, roughly speaking, has the effect of pushing the fractional curvature $Q_{2\gamma}$ to the boundary.  Two key properties we need are that the corresponding weighted GJMS operators in the interior as used in Theorem~\ref{thm:01_case} and Theorem~\ref{thm:12_case} have vanishing constant terms and that the function $\Phi$ appearing in both the statements of both results is a multiple of $Q_{2\gamma}$.  In the case $\gamma\in(0,1)$, these properties are already enough to prove the nonnegativity of $P_{2\gamma}$ when $Q_{2\gamma}$ is nonnegative.  When $\gamma\in(1,2)$, we also need to know that the adapted metric has nonnegative scalar curvature.  This is a consequence of Proposition~\ref{prop:positive_scalar_curvature} below, which states more generally that if $\gamma>1$ and $h$ has nonnegative scalar curvature, then the adapted metric has nonnegative scalar curvature which is positive in the interior.

\subsection{The simple case $\gamma\in(0,1)$}  In this case, the important properties of the adapted smooth metric measure space are easily proven.  When we discuss the \emph{adapted smooth metric measure space}, we are really discussing, for a given Poincar\'e--Einstein manifold with a choice of representative $h$ of the conformal boundary and a choice of constant $\gamma\in(0,1)$, the smooth metric measure space~\eqref{eqn:01_rhostar_smms} constructed by the following lemma.  Likewise, the \emph{adapted defining function} is the defining function $y$ constructed below and the \emph{adapted metric} is the metric $g$ of~\eqref{eqn:01_rhostar_smms}.  This result is a reformulation of~\cite[Lemma~4.5]{ChangGonzalez2011} or~\cite[Proposition~2.2]{GonzalezQing2010}, though we point out that the additional assumption $\lambda_1(-\Delta_{g_+})>\frac{n^2}{4}-\gamma^2$, which is omitted in their statements, is actually necessary and sufficient for the adapted defining function to be a well-defined positive function in all of $X$.

\begin{lem}
\label{lem:01_rhostar}
Let $(X^{n+1},M^n,g_+)$ be a Poincar\'e--Einstein manifold, fix a representative $h$ of the conformal boundary, and let $r$ denote the geodesic defining function associated to $h$.  Let $\gamma\in(0,1)$, suppose that $\lambda_1(-\Delta_{g_+})>\frac{n^2}{4}-\gamma^2$, and set $m_0=1-2\gamma$.  Then there exists a unique defining function $y$ such that
\begin{equation}
\label{eqn:01_rhostar_asymptotics}
y = r + d_\gamma^{-1}Q_{2\gamma}r^{1+2\gamma} + O(r^3)
\end{equation}
and the smooth metric measure space
\begin{equation}
\label{eqn:01_rhostar_smms}
\left(X^{n+1},g:=y^2g_+,y^{m_0}\dvol_g,m_0-1\right)
\end{equation}
satisfies $J_{\phi_0}^{m_0}=0$.
\end{lem}

\begin{proof}

Set $s=\frac{n}{2}+\gamma$ and let $v$ be the unique solution to
\begin{equation}
\label{eqn:rhostar_scattering}
-\Delta_{g_+}v - s(n-s)v = 0
\end{equation}
with $r^{s-n}v\rv_M = 1$.  Recalling that $s(n-s)\not\in\sigma(-\Delta_{g_+})$, it follows from the maximum principle (cf.\ \cite[Theorem~1]{FischerColbrieSchoen1980}) that $v>0$ if and only if $\lambda_1(-\Delta_{g_+})>s(n-s)$.  Since $s(n-s)=\frac{n^2}{4}-\gamma^2$, our assumptions imply that $v>0$.  From~\eqref{eqn:poisson_asymptotics} and~\eqref{eqn:F_expansion} it follows that
\[ v = r^{\frac{n-2\gamma}{2}} + \frac{n-2\gamma}{2d_\gamma}Q_{2\gamma}r^{\frac{n+2\gamma}{2}} + O\left(r^{\frac{n-2\gamma+4}{2}}\right) . \]
Then $y:=v^{\frac{2}{n-2\gamma}}$ has the desired asymptotic expansion~\eqref{eqn:01_rhostar_asymptotics}.  Moreover, as in the proof of Theorem~\ref{thm:01_case}, we may write~\eqref{eqn:rhostar_scattering} in terms of the smooth metric measure space $(X^{n+1},g_+,1^{m_0}\dvol,m_0-1)$ as
\[ \left(L_{2,\phi}^{m_0}\right)_+y^{\frac{m_0+n-1}{2}} = 0 . \]
Using the conformal covariance of the weighted conformal Laplacian, it follows that the smooth metric measure space~\eqref{eqn:01_rhostar_smms} is such that
\[ \frac{m_0+n-1}{2}J_{\phi_0}^{m_0} = L_{2,\phi_0}^{m_0}(1) = y^{-\frac{m_0+n+3}{2}}\left(L_{2,\phi_0}^{m_0}\right)_+\left(y^{\frac{m_0+n-1}{2}}\right) = 0 . \qedhere \]
\end{proof}

\subsection{The case $\gamma\in(1,2)$}  The adapted smooth metric measure space arises in this case by again considering~\eqref{eqn:rhostar_scattering} with $r^{s-n}v\rv_M=1$.  However, for applications we need to know some additional properties of the adapted smooth metric measure space in this case, and so we consider the following analogue of Lemma~\ref{lem:01_rhostar}.  Note that when $n=3$ and $\gamma=3/2$, the adapted defining function was used by Fefferman and Graham~\cite{FeffermanGraham2002} to study the critical $Q$-curvature, and the fact that the adapted metric has vanishing $Q$-curvature was observed and used by Chang, Qing and Yang~\cite{ChangQingYang2006} to study the renormalized volume of Poincar\'e--Einstein manifolds.  Indeed, their observations extend to all odd dimensions $n$ with the choice $\gamma=n/2$ (cf.\ Remark~\ref{rk:higher_defining_fn}).

\begin{lem}
\label{lem:12_rhostar}
Let $(X^{n+1},M^n,g_+)$ be a Poincar\'e--Einstein manifold, fix a representative $h$ of the conformal boundary, and let $r$ denote the geodesic defining function associated to $h$.  Let $\gamma\in(1,2)$ and set $m_1=3-2\gamma$.  If $\lambda_1(-\Delta_{g_+})>\frac{n^2}{4}-\gamma^2$, then there exists a unique defining function $y$ such that
\begin{equation}
\label{eqn:12_rhostar_asymptotics}
y = r - \frac{J_h}{4(\gamma-1)}r^3 + d_\gamma^{-1}Q_{2\gamma}r^{1+2\gamma} + O(r^5)
\end{equation}
and the smooth metric measure space
\begin{equation}
\label{eqn:12_rhostar_smms}
\left(X^{n+1},g:=y^2g_+,y^{m_1}\dvol_g,m_1-1\right)
\end{equation}
satisfies $Q_{\phi_1}^{m_1} = 0 = J_{\phi_0}^{m_0}$, where $J_{\phi_0}^{m_0}$ is the weighted Schouten scalar of~\eqref{eqn:01_rhostar_smms}.
\end{lem}

\begin{proof}

Consider first the case $\gamma\not=\frac{n}{2}$.  Set $s=\frac{n}{2}+\gamma$ and let $v$ be the unique solution to~\eqref{eqn:rhostar_scattering} with $r^{s-n}v\rv_M=1$. From~\eqref{eqn:poisson_asymptotics}, \eqref{eqn:F_expansion}, and~\eqref{eqn:f2} it follows that
\[ v = \left(1 - \frac{n-2\gamma}{8(\gamma-1)}J_hr^2 + O(r^4)\right)r^{\frac{n-2\gamma}{2}} + \left(\frac{n-2\gamma}{2d_\gamma}Q_{2\gamma} + O(r^2)\right)r^{\frac{n+2\gamma}{2}} . \]
Since $\lambda_1(-\Delta_{g_+})>\frac{n^2}{4}-\gamma^2$, it follows that $v>0$.  Set $y=v^{\frac{2}{n-2\gamma}}$, so that $y$ has the asymptotic expansion~\eqref{eqn:12_rhostar_asymptotics}.  Since $y$ is obtained exactly as in Lemma~\ref{lem:01_rhostar}, it follows that~\eqref{eqn:01_rhostar_smms} is such that $J_{\phi_0}^{m_0}=0$.  Using the weighted Paneitz operator in place of the weighted conformal Laplacian and computing as in the proof of Theorem~\ref{thm:12_case}, we also find that
\[ \frac{m_1+n-3}{2}Q_{\phi_1}^{m_1} = L_{4,\phi_1}^{m_1}(1) = y^{-\frac{m_1+n+5}{2}}\left(L_{4,\phi_1}^{m_1}\right)_+\left(y^{\frac{m_1+n-3}{2}}\right) = 0 , \]
which completes the proof for $\gamma\not=\frac{n}{2}$.

Suppose now that $\gamma=\frac{n}{2}$, and hence $n=3$.  Since the Poisson operator $\mP(s)$ is analytic at $s=\frac{n}{2}+\gamma=3$, it follows that the functions $v_s=\mP(s)1$ are analytic at $s=3$.  Moreover, it is clear from~\eqref{eqn:rhostar_scattering} that $v_3\equiv1$.  It follows that the functions $w_s\in C^\infty(X)$ defined by $v_s=e^{(n-s)w_s}$ are analytic at $s=3$ and satisfy $\left(w_s-\log r\right)\rv_M=0$.  Defining $y_s:=v_s^{\frac{1}{n-s}}$ as in the previous paragraph, we see that $y_s=e^{w_s}$ is analytic at $s=n$; in particular, we may define $y:=y_n$.  By analyticity, $y$ satisfies~\eqref{eqn:12_rhostar_asymptotics}, and moreover, the analyticity of $J_\phi^m$ and $Q_\phi^m$ in $m$ --- a fact easily seen from the formulae given in Section~\ref{sec:smms} --- implies, via the previous paragraph, that $J_{\phi_0}^{m_0}=0=Q_{\phi_1}^{m_1}$, as desired.
\end{proof}

\begin{remark}
\label{rk:higher_defining_fn}
Lemma~\ref{lem:12_rhostar} is in fact true for all $\gamma\in(1,\frac{n}{2}]\setminus\bN$, except that there are generally more terms which are odd powers of $r$ in the expansion~\eqref{eqn:12_rhostar_asymptotics} of order less than $1+2\gamma$.  The above proof works verbatim and, moreover, implies the vanishing of higher order weighted $Q$-curvature for the appropriately defined smooth metric measure spaces (cf.\ Theorem~\ref{thm:general_case}).
\end{remark}

As mentioned in the introduction, for $\gamma>1$ the adapted metric has an additional property that we need, namely that if the corresponding conformal representative for the conformal boundary has nonnegative scalar curvature, then the adapted metric also has nonnegative scalar curvature.  This is a generalization of Lee's result~\cite{Lee1995} relating the Yamabe constant of the conformal boundary to properties of the Laplacian of the asymptotically hyperbolic metric (see also~\cite{GuillarmouQing2010} where the relationship between scalar curvatures is mentioned explicitly).

\begin{prop}
\label{prop:positive_scalar_curvature}
Let $(X^{n+1},M^n,g_+)$ be a Poincar\'e--Einstein manifold and suppose there is a representative $h$ of the conformal boundary with nonnegative scalar curvature.  Let $\gamma\in(1,2)$ and let $y$ be the adapted defining function.  Then the scalar curvature $R_{g}$ of the adapted metric $g:=y^2g_+$ is positive in $X$.
\end{prop}

\begin{proof}

Set $y_{n+1}=v_{n+1}^{-1}$ for $v_{n+1}=\mP(n+1)1$.  By definition, this is the adapted defining function for the parameter $\frac{n}{2}+1$.  On the other hand, using the assumption $R_h\geq0$, Lee showed~\cite{Lee1995} (see also~\cite[Lemma~2.2]{GuillarmouQing2010}) that the adapted metric $g_{n+1}:=y_{n+1}^2g_+$ has positive scalar curvature in $X$.

We now use the continuity method.  Let $I$ denote the set of all $s\in\left[\frac{n}{2}+\gamma,n+1\right]$ such that the scalar curvature $R_s$ of $g_s$ is positive for $g_s$ the adapted metric associated to $(X^{n+1},M^n,g_+)$ and the parameter $s-\frac{n}{2}$.  We have already seen that $n+1\in I$.  Hence it suffices to show that $I$ is open and closed.

Since $R_h\geq0$, it holds that $\lambda_1(-\Delta_{g_+})=\frac{n^2}{4}$; see~\cite{Lee1995}.  Hence the Poisson operator $\mP(s)$ is analytic for $s\in(\frac{n}{2},\infty)$; see~\cite{GrahamZworski2003}.  This guarantees that the adapted defining functions $y_s$ associated to $s\in\left[\frac{n}{2}+\gamma,n+1\right]$ are analytic in $s$, and hence $g_s:=y_s^2g_+$ form an analytic family of metrics in $X$.  In particular, $I$ is open.

Suppose that $s$ is a limit point of $I$.  Then $R_s\geq0$.  From Lemma~\ref{lem:12_rhostar} we have that $Q_{\phi_1}^{m_1}=0$ and $J_{\phi_0}^{m_0}=0$ for the adapted smooth metric measure spaces~\eqref{eqn:12_rhostar_smms} and~\eqref{eqn:01_rhostar_smms}, respectively, determined by the parameter $s-\frac{n}{2}$.  By Lemma~\ref{lem:pe_smms_formulae}, the latter condition implies that $J_{\phi_1}^{m_1}=\frac{2}{2s-n-1}J_s$, where $J_s=\frac{R_s}{2n}$ is the trace of the Schouten tensor $P_s$ of $g_s$.  On the other hand, the former condition implies that
\begin{equation}
\label{eqn:vanishing_q2gamma}
\frac{1}{2s-n-1}\left(\Delta_{g_s} + m_1y_s^{-1}\frac{\partial}{\partial y_s}\right)J_s = -\lv P_s\rv_{g_s}^2 + \frac{n+1}{(2s-n-1)^2}J_s^2 .
\end{equation}
If $s\not\in I$, then there is a point $x\in X$ such that $J_s(x)=0$, and hence, by the strong maximum principle, $J_s\equiv0$.  By arguing as in~\cite[Lemma~2.2]{GuillarmouQing2010}, we see that this contradicts the fact $\lambda_1(-\Delta_{g_+})=\frac{n^2}{4}$.  Thus $s\in I$; i.e.\ $I$ is closed.
\end{proof}

\begin{remark}
The above proof shows that if $(X^{n+1},M^n,g_+)$ is a Poincar\'e--Einstein manifold with conformal boundary having nonnegative Yamabe constant, then for any $\gamma>1$ the adapted metric has positive scalar curvature in $X$.
\end{remark}

In addition to giving a sign on the scalar curvature of the adapted metric, Proposition~\ref{prop:positive_scalar_curvature} leads to a strong rigidity result for Poincar\'e--Einstein manifolds which admit a scalar flat representative of the conformal boundary.  By way of motivation, recall that if $(M^n,g)$ is scalar flat, then the fourth-order $Q$-curvature is nonpositive and vanishes identically if and only if $(M^n,g)$ is Ricci flat.  For Poincar\'e--Einstein manifolds, the same result is true when replacing the fourth-order $Q$-curvature by $Q_{2\gamma}$ for any $\gamma\in(1,2)$, though the rigid case cannot happen.  As a first step in this direction, we observe that scalar flat representatives of Poincar\'e--Einstein manifolds have nonpositive $Q_{2\gamma}$ for $\gamma\in(1,2)$.

\begin{cor}
\label{cor:signs_of_r_and_q}
Let $(X^{n+1},M^n,g_+)$ be a Poincar\'e--Einstein manifold and suppose there is a scalar-flat representative $h$ of the conformal boundary.  Then for any $\gamma\in(1,2)$, it holds that $Q_{2\gamma}\leq0$.  Moreover, if $Q_{2\gamma}\equiv0$, then $h$ is Ricci flat.
\end{cor}

\begin{proof}

Let $g=y^2g_+$ be the adapted metric.  It follows from Lemma~\ref{lem:pe_smms_formulae} and the fact $J_{\phi_0}^{m_0}=0$ that
\begin{equation}
\begin{split}
\label{eqn:12_J_asymptotics}
J_{g} & = \frac{2\gamma-1}{2}\left(y^{-2} - \lv dy^{-1}\rv_{g_+}^2\right) \\
& = \frac{2\gamma-1}{2(\gamma-1)}J_h - \frac{2\gamma(2\gamma-1)}{d_\gamma}Q_{2\gamma}r^{2\gamma-2} + O(r^2),
\end{split}
\end{equation}
where the second equality uses the expansion~\eqref{eqn:12_rhostar_asymptotics} of $y$ near $M$.  By Proposition~\ref{prop:positive_scalar_curvature} we have that $J_{g}\geq0$.  Since $J_h\equiv0$ and $d_\gamma>0$ for $\gamma\in(1,2)$, it follows that $Q_{2\gamma}\leq0$.

Suppose now that $Q_{2\gamma}\equiv0$.  As pointed out to us by Fang Wang, following the computations outlined in~\cite[Proposition~4.2]{GrahamZworski2003} easily leads to
\[ v := \mP\left(\frac{n}{2}+\gamma\right)1 = r^{\frac{n-2\gamma}{2}} - \frac{n-2\gamma}{32(\gamma-2)}\lv P\rv_h^2r^{\frac{n-2\gamma+8}{2}} + O\left(\rho^{\frac{n-2\gamma+12}{2}}\right) \]
for $P$ the Schouten tensor of $h$.  In particular, it follows again from the conformal transformation formula for the scalar curvature that
\[ J_{g} = \frac{2\gamma-1}{4(\gamma-2)}\lv P\rv_h^2 r^2 + O\left(r^4\right) . \]
Thus, since $\gamma\in(1,2)$ and $J_{g}\geq0$, we see that $P\equiv0$; i.e.\ $h$ is Ricci flat.
\end{proof}

As stated above, the rigid case of Corollary~\ref{cor:signs_of_r_and_q} does not occur.  This is a feature of the compactness requirement in our definition of a Poincar\'e--Einstein manifold.

\begin{cor}
\label{cor:signs_of_r_and_q2}
Let $(X^{n+1},M^n,g_+)$ be a Poincar\'e--Einstein manifold and suppose there is a scalar-flat representative $h$ of the conformal boundary.  Then for all $\gamma\in(1,2)$, there exists a point $p\in M$ such that $Q_{2\gamma}(p)<0$.
\end{cor}

\begin{proof}

Suppose to the contrary that there is a $\gamma\in(1,2)$ such that $Q_{2\gamma}\equiv0$.  From Corollary~\ref{cor:signs_of_r_and_q}, the representative $h$ is Ricci flat.  From~\cite[Theorem~A]{ChruscielDelayLeeSkinner2005} and~\cite[Theorem~4.8]{FeffermanGraham2012} we conclude that, in a collar neighborhood $V=[0,\varepsilon)\times M$ of $M$, we can write $g_+=r^{-2}(dr^2+g_r)$ for
\[ g_r = h + kr^n + o(r^n), \]
where $k$ is a symmetric $(0,2)$-tensor on $M$ such that $\tr_h k=0$ and $\divsymb k=0$ and the higher order terms in $g_r$ are determined by $h$ and $k$.  Indeed, from the formulae~\cite[(2.4) and (2.6)]{GrahamHirachi2004}, we see that
\begin{align*}
0 & = rg^{lm}\partial_r^2g_{lm} - g^{lm}\partial_r g_{lm} - \frac{r}{2}g^{lp}g^{mq}\partial_r g_{lm}\partial_r g_{pq} \\
0 & = r\partial_r^2g_{ij} - (n-1)\partial_rg_{ij} - rg^{lm}\partial_rg_{il}\partial_rg_{jm} + \frac{r}{2}g^{lm}\partial_rg_{lm}\partial_rg_{ij} - g^{lm}(\partial_rg_{lm})g_{ij} .
\end{align*}
By differentiating up to $2n-1$ times and evaluating at $r=0$, we readily compute that in fact
\begin{equation}
\label{eqn:ricciflat_g}
g_r = h + kr^n + \left(\frac{1}{2}(k^2)_0 + \frac{3n-4}{8n(n-1)}\lv k\rv^2 h\right)r^{2n} + o(r^{2n}) ,
\end{equation}
where $(k^2)_0$ is the tracefree part of the composition $(k^2)_{ij} = k_{il}k_j^l$.

Next, we compute the formal solution to
\begin{equation}
\label{eqn:pe_rf}
\begin{cases}
-\Delta_{g_+}v - s(n-s)v = 0, \\
v = Fr^{n-s}, & F\in C^\infty(\oX), \\
F\rv_{r=0} = 1 .
\end{cases}
\end{equation}
Since $k$ is trace-free, we find that
\begin{align*}
\sqrt{\det g_+} & = r^{-n-1}\sqrt{\det h}\left( 1 - \frac{n}{16(n-1)}\lv k\rv^2 r^{2n} + o(r^{2n})\right) \\
g^{ij} & = h^{ij} - r^nk^{ij} + O(r^{2n}) .
\end{align*}
It thus follows that
\[ \Delta_{g_+} = \left(r\partial_r\right)^2 - nr\partial_r - \frac{n^2}{8(n-1)}\lv k\rv^2 r^{2n+1}\partial_r + r^2\Delta_h - r^{n+2}\delta k d + O(r^{2n+2}) . \]
In particular
\begin{align*}
\left(\Delta_{g_+}+s(n-s)\right)r^{n-s} & = -\frac{n^2(n-s)}{8(n-1)}\lv k\rv^2 r^{3n-s} + o(r^{3n-s}) \\
\left(\Delta_{g_+}+s(n-s)\right)\left(fr^{3n-s}\right) & = 2n(3n-2s)fr^{3n-s} + o(r^{3n-s}) 
\end{align*}
for any $f\in C^\infty(M)$, from which it immediately follows that the solution to~\eqref{eqn:pe_rf} is
\begin{equation}
\label{eqn:pr_rf_odd_soln}
v = \left( 1 + \frac{n(n-s)}{16(n-1)(3n-2s)}\lv k\rv^2 r^{2n} + o(r^{2n})\right) r^{n-s} .
\end{equation}

Finally, since $Q_{2\gamma}=0$, it follows that the solution to $-\Delta_+v-s(n-s)v=0$ with $s=\frac{n}{2}+\gamma$ and $r^{s-n}v\to1$ as $r\to0$ is an exact solution of~\eqref{eqn:pe_rf}.  Set $y^{n-s}=v$, so that
\begin{equation}
\label{eqn:pr_rf_odd_soln_y}
y = r + \frac{n}{16(n-1)(3n-2s)}\lv k\rv^2 r^{2n+1} + o(r^{2n+1}) .
\end{equation}
We then compute that
\begin{align*}
y^{-2} - \lv dy^{-1}\rv_{g_+}^2 & = y^{-2} - \left(r\frac{\partial y^{-1}}{\partial r}\right)^2 - r^2\lv\nabla_{g_r} y\rv_{g_r}^2 \\
& = -\frac{n^2}{4(n-1)(3n-2s)}\lv k\rv^2 r^{2n-2} + o(r^{2n-2}) .
\end{align*}
Since $3n-2s = 2(n-\gamma)>0$, this is nonpositive near the boundary and becomes negative somewhere unless $k\equiv0$.  On the other hand, the scalar curvature
\[ R_{y^2g_+} = n(2s-n-1)\left(y^{-2}-\lv dy^{-1}\rv_{g_+}^2\right) . \]
By Proposition~\ref{prop:positive_scalar_curvature}, this is nonnegative, and hence $k\equiv0$.  Thus $g_r=h\mod O(r^\infty)$ in $V$.  But $r^{-2}(dr^2\oplus h)$ is an Einstein metric in $V$, so Biquard's unique continuation theorem~\cite{Biquard2008} implies that $g_r=h$ in $V$.  Hence $v=r^{n-s}$ solves $-\Delta v-s(n-s)v=0$ in $V$, and so, by uniqueness of solutions to the Poisson equation~\cite{GrahamZworski2003}, it holds that $y=r$ in $V$.  A direct computation then shows that $R_{y^2g_+}\equiv0$ in $V$, which contradicts Proposition~\ref{prop:positive_scalar_curvature}.
\end{proof}

One can also interpret Corollary~\ref{cor:signs_of_r_and_q2} as stating that there does not exist a Poincar\'e--Einstein manifold which admits a Ricci flat representative of the conformal boundary with $Q_{2\gamma}\equiv0$ for some $\gamma\in(1,2)$; in fact, the proof works for $\gamma\in(1,n)\setminus\bN$.  The following example, which is inspired by~\cite[Example~9.118(d)]{Besse}, shows that the nonlocal assumption $Q_{2\gamma}\equiv0$ is necessary; i.e.\ there do exist Poincar\'e--Einstein manifolds with Ricci flat representatives of the conformal boundary.

\begin{prop}
\label{prop:nonuniqueness}
Let $(F^{n-1},g_F)$ be a compact Ricci flat manifold.  Define the metric $g_+$ on $\bR^2\times F^{n-1}$ by
\[ g_+ = dt^2 \oplus \left(\frac{2}{n}\cosh^{\frac{2-n}{n}}(\frac{n}{2}t)\sinh(\frac{n}{2}t)\right)^2d\theta^2 \oplus \cosh^{\frac{4}{n}}(\frac{n}{2}t)g_F , \]
where $(t,\theta)$ are polar coordinates on $\bR^2$.  Then $(\bR^2\times F^{n-1},S^1\times F^{n-1},g_+)$ is a Poincar\'e--Einstein manifold.  Moreover, the metric
\begin{equation}
\label{eqn:nonuniqueness_h}
h := \left(\frac{2}{n}\right)^2d\theta^2 \oplus g_F
\end{equation}
is a Ricci flat metric with fractional $Q$-curvature
\begin{equation}
\label{eqn:nonuniqueness_q}
Q_{2\gamma} = 2^{\frac{2(n-2)}{n}\gamma} \frac{\Gamma(\gamma)\Gamma\left(-\frac{\gamma}{n}\right)\Gamma\left(\frac{n+2\gamma}{2n}\right)}{n\Gamma(-\gamma)\Gamma\left(\frac{\gamma}{n}\right)\Gamma\left(\frac{3n-2\gamma}{2n}\right)} .
\end{equation}
for any $\gamma>0$.
\end{prop}

\begin{proof}

It is straightforward to compute that $\Ric(g_+)=-ng_+$.  Setting $r=2^{2/n}e^{-t}$, we see that
\[ g_+ = r^{-2}\left[ dr^2\oplus\left(\frac{2}{n}\left(1+\frac{r^n}{4}\right)^{\frac{2-n}{n}}\left(1-\frac{r^n}{4}\right)\right)^2d\theta^2 \oplus\left(1+\frac{r^n}{4}\right)^{\frac{4}{n}}g_F \right] , \]
from which it immediately follows that $(\bR^2\times F^{n-1},S^1\times F^{n-1},g_+)$ is Poincar\'e--Einstein and that $h$ is a representative of the conformal boundary.  Clearly $h$ is Ricci flat.

To compute the fractional $Q$-curvature of $h$, note that for $m_0=1-2\gamma$, the weighted conformal Laplacian $L_{2,\phi_0}^{m_0}$ of $(\bR^2\times F^{n-1},r^2g_1,r^{m_0}\dvol,m_0-1)$ acts on functions $U=U(r)$ as
\[ L_{2,\phi_0}^{m_0}U = -\partial_r^2U - \frac{16m_0-(m_0+2n)r^{2n}}{r(16-r^{2n})}\partial_r U + \frac{n(m_0+n-1)r^{2n-2}}{16-r^{2n}}U . \]
By making the change of variables $16x=r^{2n}$, we see that $L_{2,\phi_0}^{m_0}U=0$ if and only if
\[ x(1-x)\partial_x^2U + \left(\frac{2n+m_0-1}{2n} - \frac{m_0+4n-1}{2n}x\right)\partial_xU - \frac{m_0+n-1}{4n}U = 0 . \]
This is a standard hypergeometric ODE, and it is readily computed (cf.\ \cite[Chapter~15]{AbramowitzStegun}) that the unique solution which is regular at $x=1$ and satisfies $U(0)=1$ is
\begin{equation}
\label{eqn:soln}
\begin{split}
U(r) & = {}_2F{}_1\left(\frac{n-2\gamma}{2n},\frac{1}{2};\frac{n-\gamma}{n};\frac{r^{2n}}{16}\right) \\
& \quad + 2^{-\frac{4\gamma}{n}}\frac{\Gamma\left(-\frac{\gamma}{n}\right)\Gamma\left(\frac{n+2\gamma}{2n}\right)}{\Gamma\left(\frac{\gamma}{n}\right)\Gamma\left(\frac{n-2\gamma}{2n}\right)}r^{2\gamma} {}_2F_1\left(\frac{n+2\gamma}{2n},\frac{1}{2};\frac{n+\gamma}{n};\frac{r^{2n}}{16}\right) .
\end{split}
\end{equation}
As observed in the proof of Theorem~\ref{thm:01_case}, the function $v(r)=r^{n-s}U(r)$ satisfies~\eqref{eqn:poisson_equation} with $s=\frac{n}{2}+\gamma$ and $F\rv_{r=0}=1$.  The formula~\eqref{eqn:nonuniqueness_q} then follows immediately from~\eqref{eqn:soln} and the definition of $Q_{2\gamma}$.
\end{proof}
\section{Positivity results for the fractional GJMS operators}
\label{sec:positivity}

In this section we consider two types of positivity results for the fractional GJMS operators $P_{2\gamma}$ with $\gamma\in(0,2)$ under assumptions on their zeroth order terms $Q_{2\gamma}$ and, in the case $\gamma\in(1,2)$, the scalar curvature.  While the conclusion are conformally invariant, our assumptions on $Q_{2\gamma}$ and $R$ depend on the choice of representative of the conformal boundary.

One type of result we prove is the positivity of the first eigenvalue of $P_{2\gamma}$.  When $\gamma\in(0,1)$, this result is due to Gonz\'alez and Qing~\cite{GonzalezQing2010}.  When $\gamma\in(1,2)$ this result is Theorem~\ref{thm:positive}.  In this case, our result is similar in spirit to the corresponding result of Theorem~\ref{thm:positive} for the Paneitz operator proven by Gursky~\cite{Gursky1999} (when $n=4$), by Xu and P.\ Yang~\cite{XuYang2001} (when $n\geq6$) and by Gursky and Malchiodi~\cite{GurskyMalchiodi2014} (when $n\geq5$).  The basic idea underlying our proofs is to exhibit the energy $(P_{2\gamma}f,f)$ as the sum of $\int Q_{2\gamma}f^2$ and the energy of the corresponding weighted GJMS operator in the interior from Theorem~\ref{thm:01_case} or Theorem~\ref{thm:12_case}.  When $\gamma>1$, this involves studying the energy of the weighted Paneitz operator, and our method for proving its nonnegativity is analogous to the method used in~\cite{Gursky1999,GurskyMalchiodi2014,XuYang2001}.

The other type of result we prove is a strong maximum principle for $P_{2\gamma}$.  Upon combining our results below with a trivial observation for the conformal Laplacian and a recent result of Gursky and Malchiodi~\cite[Theorem~A]{GurskyMalchiodi2014}, we have that for any $\gamma\in(0,2]$, if $P_{2\gamma}(1)$ is \emph{semi-positive} --- i.e.\ if $P_{2\gamma}(1)\geq0$ and is not identically zero --- then for any $f\in C^\infty(M)$ such that $P_{2\gamma}f\geq0$, either $f>0$ or $f\equiv0$.

In the case of the conformal Laplacian $P_2=-\Delta+\frac{n-2}{4(n-1)}R_h$ on $(M^n,h)$ this is a simple consequence of the strong maximum principle.  For all other values $\gamma\in(0,2]\setminus\{1\}$, this result requires more work.  When $\gamma\in(0,1)$, one must combine the strong maximum principle in the interior with a Hopf Lemma on the boundary for the degenerate elliptic operator $L_{2,\phi_0}^{m_0}$ to derive the conclusion.  When $\gamma\in(1,2)$, the interior operator is fourth order (cf.\ Theorem~\ref{thm:12_case}), and so we lack even a maximum principle on the interior.  Here we overcome the difficulty using the idea of Gursky and Malchiodi~\cite{GurskyMalchiodi2014}: roughly speaking, they use conformal covariance and relations between the scalar curvature and the (fourth-order) $Q$-curvature to reduce the condition $P_4u\geq0$ to a nonnegativity condition involving a strongly elliptic second order operator.  The same idea will work to handle the cases $\gamma\in(1,2)$, though again we will also need to appeal to a Hopf Lemma on the boundary.

The Hopf Lemma we need does not seem to appear in the literature, though the following mild generalization of the Hopf Lemma proven by Gonz\'alez and Qing~\cite[Theorem~3.5]{GonzalezQing2010} suffices for our needs.  Our result is only more general in that it allows for degenerate elliptic operators with nonvanishing constant terms and it does not require one to consider the adapted metric.  We have opted to state our result for a large class of asymptotically hyperbolic metrics, but have made no attempt to find the most general statement.

\begin{prop}
\label{prop:hopf}
Let $(X^{n+1},g,\rho^m\dvol,m-1)$ be a compact smooth metric measure space with $m\in(-1,1)$ and $M:=\partial X\not=\emptyset$ such that, in a collar neighborhood $\widetilde{M}=[0,\varepsilon)\times M$ of $M$,
\[ \rho = r\left(1+\Phi r^{1-m} + \rho_{(1)} r + o(r)\right), \quad\text{and}\quad g = \left(\frac{\rho}{r}\right)^2\left(dr^2 + h + rh_{(1)} + o(r)\right) \]
for $r=d(\cdot,M)$ the distance in $X$ to the boundary $M$ and $\rho_{(1)},\Phi\in C^\infty(M)$ and both $h_{(1)}$ and the terms of order $o(r)$ symmetric $(0,2)$-tensors on $M$.  Suppose that $U\in C^\infty(X)\cap C^0(\oX)$ is a nonnegative function such that
\begin{equation}
\label{eqn:hopf_interior}
-\Delta_\phi U + \psi U \geq 0 \quad\text{in $X$}
\end{equation}
for $\psi\in C^\infty(X)\cap C^0(\oX)$.  If there is a point $q_0\in M$ and a constant $s_0>0$ sufficiently small such that
\begin{enumerate}
\item $U(q_0)=0$ and
\item there is a point $p\in M$ such that $q_0\in\Gamma_{s_0}^0\setminus\overline{\Gamma_{s_0/2}^0}$ and $U>0$ on $\partial\Gamma_{s_0/2}^0$ for
\[ \Gamma_s^0 = \left\{ x \in M \colon d(x,p) < s \right\} , \]
\end{enumerate}
then
\begin{equation}
\label{eqn:hopf_conclusion}
\lim_{\rho\to0} \rho^m\frac{\partial U}{\partial\rho}\left(q_0,\rho\right) > 0 .
\end{equation}
\end{prop}

\begin{proof}

The proof by Gonz\'alez and Qing~\cite{GonzalezQing2010} of their version of the Hopf Lemma establishes the result with only minor modifications.  We include a sketch of the proof for the convenience of the reader, and refer to~\cite[p.\ 1550]{GonzalezQing2010} for further details of the computations.

Denote by $L$ the operator $L=-\Delta_\phi U + \psi$.  It is clear that the strong maximum principle applies in the interior, and hence $U>0$ in $X$.  It is convenient to separate the rest of the proof into two cases.

\emph{Case 1: $m<0$ or $\Phi=0$}.  Consider the test function
\[ W=\rho^m\left(\rho+A\rho^2\right)\left(e^{-Bs}-e^{-Bs_0}\right) \]
for $A,B$ large constants to be chosen later and $s=d_M(\cdot,p)$ the distance in $M$ to the point $p$.  A straightforward computation (cf.\ \cite{GonzalezQing2010}) shows that
\begin{align*}
\rho^m\Delta_\phi W & = \left((2-m)A + c\rho_{(1)} + O(\rho)\right)\left(e^{-Bs}-e^{-Bs_0}\right) \\
& \quad + \left(\rho + O(\rho^2)\right)\left(-\frac{nB}{s}+B^2+o(1)\right)e^{-Bs},
\end{align*}
where $c$ is a constant depending only on $n$ and the final $o(1)$ denotes terms which are small in $\rho$ and $s$.  It follows that, for $A,B$ sufficiently large, $L(W)\leq0$.  It follows from this, the definition of $W$, and the assumptions on $U$ that there is a constant $\varepsilon>0$ such that $L(U-\varepsilon W)\geq 0$ and $U-\varepsilon W\geq0$ on $\partial\left(\Gamma_{s_0}^0\setminus\overline{\Gamma_{s_0/2}^0}\times(0,s_0)\right)$.  Thus $U-\varepsilon W>0$ in $X$, from which we conclude that
\[ \lim_{\rho\to0} \rho^m\frac{\partial}{\partial \rho}\left(U-\varepsilon W\right)(q_0,\rho) \geq 0 . \]
A simple computation shows that
\[ \lim_{\rho\to0} \rho^m\frac{\partial W}{\partial \rho}(q_0,\rho) = (1-m)\left(e^{-Bd(p_0,q_0)}-e^{-Bs_0}\right) > 0, \]
from which the conclusion~\eqref{eqn:hopf_conclusion} immediately follows.

\emph{Case 2: $m\geq0$}.  Consider instead the test function
\[ W = \rho^{-m}\left(\rho+A\rho^{2-m}\right)\left(e^{-Bs}-e^{-Bs_0}\right) . \]
Computing again as in~\cite{GonzalezQing2010} shows that
\begin{align*}
\rho^m\Delta_\phi W & = \left((1-m)(2-m)A\rho^{-m} + c\Phi \rho^{-m} + o(\rho^{-m})\right)\left(e^{-Bs}-e^{-Bs_0}\right) \\
& \quad + \left(\rho + o(\rho)\right)\left(-\frac{nB}{s}+B^2+o(1)\right)e^{-Bs}.
\end{align*}
As in the previous case, we may choose constants $A$ and $B$ sufficiently large and constants $s_0,\varepsilon>0$ sufficiently small so that $L(U-\varepsilon W)\geq0$ and $U-\varepsilon W\geq0$ on $\partial\left(\Gamma_{s_0}^0\setminus\overline{\Gamma_{s_0/2}^0}\times(0,s_0\right)$.  From this point the conclusion follows exactly as in the previous paragraph.
\end{proof}

\subsection{The simple case $\gamma\in(0,1)$}\label{subsec:positivity/01}  In order to motivate our approach, let us first rederive using smooth metric measure spaces a result of Gonz\'alez and Qing~\cite{GonzalezQing2010} on sufficient conditions for the positivity of the first eigenvalue of $P_{2\gamma}$.  To that end, note that an immediate corollary of Theorem~\ref{thm:01_case} and Lemma~\ref{lem:01_rhostar} is the following extension formula for $P_{2\gamma}$.

\begin{lem}
\label{lem:01_rhostar_extension}
Let $(X^{n+1},M^n,g_+)$ be a Poincar\'e--Einstein manifold and fix a representative $h$ of the conformal boundary.  Let $\gamma\in(0,1)$ be such that $\lambda_1(-\Delta_{g_+})>\frac{n^2}{4}-\gamma^2$.  Set $m_0=1-2\gamma$ and let $(\oX^{n+1},g,y^{m_0}\dvol,m_0-1)$ be the adapted smooth metric measure space~\eqref{eqn:01_rhostar_smms}.  Then for each $f\in C^\infty(M)$, the solution $U$ to the boundary value problem
\begin{equation}
\label{eqn:01_extension_y}
\begin{cases}
-\Delta_{\phi_0} U = 0, &\quad\text{in $X^{n+1}$}, \\
U=f, &\quad\text{on $M$},
\end{cases}
\end{equation}
is such that
\begin{equation}
\label{eqn:01_extension_defn_y}
P_{2\gamma}f = \frac{n-2\gamma}{2}Q_{2\gamma}f + \frac{d_\gamma}{2\gamma}\lim_{y\to0}y^{m_0}\frac{\partial U}{\partial y} .
\end{equation}
\end{lem}

The conclusion~\eqref{eqn:01_extension_defn_y} immediately leads to an energy identity relating $P_{2\gamma}$ and the weighted conformal Laplacian of the interior which, by the properties of the adapted smooth metric measure space, yields the following result from~\cite{GonzalezQing2010}.  Note that the assumption $\lambda_1(-\Delta_{g_+})>\frac{n^2}{4}-\gamma^2$ is required because of our use of the adapted smooth metric measure space; it does not seem to be known whether this assumption is necessary.

\begin{thm}[Gonz\'alez--Qing~\cite{GonzalezQing2010}]
\label{thm:gonzalez_qing}
Let $(X^{n+1},M^n,g_+)$ be a Poincar\'e--Einstein manifold and let $\gamma\in(0,1)$.  Suppose that there exists a representative $h$ of the conformal boundary such that $Q_{2\gamma}\geq0$.  Suppose additionally that $\lambda_1(-\Delta_{g_+})>\frac{n^2}{4}-\gamma^2$.  Then $P_{2\gamma}\geq0$.  Moreover, $\ker P_{2\gamma}\not=\{0\}$ if and only if $Q_{2\gamma}\equiv0$, in which case $\ker P_{2\gamma}=\bR$ is the space of constant functions.
\end{thm}

\begin{proof}

Let $(X^{n+1},g,y^{m_0}\dvol,m_0-1)$ be the adapted smooth metric measure space.  Given $f\in C^\infty(M)$ not identically zero, let $U$ be the solution to~\eqref{eqn:01_extension_y}.  It follows from~\eqref{eqn:01_extension_defn_y} and integration by parts that
\[ \int_M P_{2\gamma}f\,f\dvol_h = \frac{n-2\gamma}{2}\int_M Q_{2\gamma}f^2\dvol_h - \frac{d_\gamma}{2\gamma}\int_X \lv\nabla U\rv^2\,y^{m_0}\dvol_{g} . \]
Since $\gamma\in(0,1)$, we see from~\eqref{eqn:scattering_definition} that $d_\gamma<0$.  Hence $\int P_{2\gamma}f\,f\geq0$, and moreover equality holds if and only if $U$ (and hence $f$) is constant and $Q_{2\gamma}\equiv0$, as desired.
\end{proof}

Using Proposition~\ref{prop:hopf} (the Hopf Lemma in~\cite{GonzalezQing2010} would also suffice), we can also prove the following strong maximum principle for the fractional GJMS operators $P_{2\gamma}$ with $\gamma\in(0,1)$ and $Q_{2\gamma}$ semi-positive.  Since we use the adapted smooth metric measure space, we again need to make a spectral assumption on $-\Delta_{g_+}$.

\begin{thm}
\label{thm:01_maximum}
Let $(X^{n+1},M^n,g_+)$ be a Poincar\'e--Einstein manifold.  Let $\gamma\in(0,1)$ be such that $\lambda_1(-\Delta_{g_+})>\frac{n^2}{4}-\gamma^2$ and suppose that there is a representative $h$ of the conformal boundary with $Q_{2\gamma}$ semi-positive.  Then for any $f\in C^\infty(M)$ such that $P_{2\gamma}f\geq0$, either $f>0$ or $f\equiv0$.
\end{thm}

\begin{proof}

Suppose $f\not\equiv0$ and consider the adapted smooth metric measure space $(X^{n+1},g,y^{m_0}\dvol,m_0-1)$.  Let $U\in C^\infty(X)$ be the solution to the boundary value problem~\eqref{eqn:01_extension_y}.  By the strong maximum principle, either $U$ is constant or the minimum of $U$ occurs on $M$.  If $U$ is constant, then $U\equiv f$, and hence~\eqref{eqn:01_extension_defn_y} implies that $f>0$.  On the other hand, if $U$ is not constant, it follows from~\eqref{eqn:01_extension_defn_y} that at a point $p\in M$ which minimizes $U$,
\begin{equation}
\label{eqn:01_maximum_min}
0 \leq \frac{n-2\gamma}{2}Q_{2\gamma}\min U + \frac{d_\gamma}{2\gamma}\lim_{y\to0}y^{m_0}\frac{\partial U}{\partial y}(p,y) \leq \frac{n-2\gamma}{2}Q_{2\gamma}\min U
\end{equation}
(recall from~\eqref{eqn:scattering_definition} that $d_\gamma<0$ for $\gamma\in(0,1)$).  Thus $\min U\geq0$.  Moreover, if $\min U=0$, then Proposition~\ref{prop:hopf} implies that the second inequality in~\eqref{eqn:01_maximum_min} is strict, a contradiction.  Hence $\min f=\min U>0$, as desired.
\end{proof}

\subsection{The case $\gamma\in(1,2)$.}  We now turn to the proofs of Theorem~\ref{thm:positive} and Theorem~\ref{thm:maximum}, which are the analogues for $\gamma\in(1,2)$ of Theorem~\ref{thm:gonzalez_qing} and Theorem~\ref{thm:01_maximum}, respectively.  In both cases, the ideas of the proof are similar in spirit to their counterparts from Subsection~\ref{subsec:positivity/01}, though they are technically more difficult due to the fact that when $\gamma\in(1,2)$, the fractional GJMS operator $P_{2\gamma}$ is the boundary operator of a fourth order operator, the weighted Paneitz operator.  A key fact used in overcoming these difficulties is the following result establishing a relationship between the scalar curvature and the fractional $Q$-curvature $Q_{2\gamma}$ with $\gamma\in(1,2)$ of a representative for the conformal boundary (cf.\ \cite[Lemma~2.1]{GurskyMalchiodi2014}).  It is here that we require our Hopf Lemma.

\begin{lem}
\label{lem:12_maximum}
Let $(X^{n+1},M^n,g_+)$ be a Poincar\'e--Einstein manifold.  Let $\gamma\in(1,2)$ and suppose that there is a representative $h$ for the conformal boundary with $R_h$ and $Q_{2\gamma}$ nonnegative, at least one of which is semi-positive.  Then $R_h>0$.
\end{lem}

\begin{proof}

Let $y$ be the adapted defining function.  In particular, the vanishing of the weighted $Q$-curvature of $(X^{n+1},y^2g_+,y^{m_1}\dvol,m_1-1)$ for $m_1=3-2\gamma$ and the asymptotics~\eqref{eqn:12_J_asymptotics} of $J=\frac{1}{2n}R_{\rho^2g_+}$ imply that (cf.\ \eqref{eqn:qphi_zero})
\begin{align}
\label{eqn:12_maximum_r_interior} 0 & \leq -\frac{1}{2\gamma-1}\Delta_\phi J + \frac{(n+2-2\gamma)(n+2\gamma)}{(1-2\gamma)^2(n+1)}J^2,& \quad\text{in $X$}, \\
\label{eqn:12_maximum_r_boundary} J & = \frac{2\gamma-1}{4(n-1)(\gamma-1)}R_h,& \quad\text{on $M$}, \\
\label{eqn:12_maximum_r_neumann} \lim_{y\to0}y^{m_1}\frac{\partial J}{\partial y} & = -\frac{4\gamma(2\gamma-1)(\gamma-1)}{d_\gamma}Q_{2\gamma} .
\end{align}
By Proposition~\ref{prop:positive_scalar_curvature}, we know that $J>0$ in $X$.  Recall from~\eqref{eqn:scattering_definition} that $d_\gamma>0$ for $\gamma\in(1,2)$.  If $R_h\equiv0$, combining~\eqref{eqn:12_maximum_r_neumann} and the assumption that $Q_{2\gamma}$ is semi-positive shows that there are points $p\in X$ such that $J(p)<0$, a contraction.  It follows that if there is a point $q\in M$ such that $R_h(q)=0$, then we can apply Proposition~\ref{prop:hopf} to $J$ and conclude that
\[ \lim_{y\to0}y^{m_1}\frac{\partial J}{\partial y}(q,y) > 0, \]
which again contradicts~\eqref{eqn:12_maximum_r_neumann}.  Hence $R_h>0$ on $M$, as desired.
\end{proof}

We first move towards the proof of Theorem~\ref{thm:positive}, which establishes a necessary condition for the first eigenvalue of $P_{2\gamma}$ to be positive.  One key ingredient is the following lemma, which is an immediate consequence of Theorem~\ref{thm:12_case} and Lemma~\ref{lem:12_rhostar}.

\begin{lem}
\label{lem:12_rhostar_extension}
Let $(X^{n+1},M^n,g_+)$ be a Poincar\'e--Einstein manifold and fix a representative $h$ of the conformal boundary.  Let $\gamma\in(1,2)$ be such that $\lambda_1(-\Delta_{g_+})>\frac{n^2}{4}-\gamma^2$.  Set $m_1=3-2\gamma$ and let $(X^{n+1},g,y^{m_1}\dvol,m_1-1)$ be the adapted smooth metric measure space~\eqref{eqn:12_rhostar_smms}.  Then for each $f\in C^\infty(M)$, the solution $U$ to the boundary value problem
\begin{equation}
\label{eqn:12_extension_y}
\begin{cases}
L_{4,\phi_1}^{m_1}U = 0, &\quad\text{in $X^{n+1}$}, \\
U=f, &\quad\text{on $M$}, \\
\lim_{y\to0}y^{m_0}\frac{\partial}{\partial y}U = 0
\end{cases}
\end{equation}
is such that
\begin{equation}
\label{eqn:12_extension_defn_y}
P_{2\gamma}f = \frac{n-2\gamma}{2}Q_{2\gamma}f + \frac{d_\gamma}{8\gamma(\gamma-1)}\lim_{y\to0} y^{m_1}\frac{\partial}{\partial y}\Delta_{\phi_1}U .
\end{equation}
\end{lem}

As suggested by the proof of Theorem~\ref{thm:gonzalez_qing}, the remaining ingredient in the proof of Theorem~\ref{thm:positive} is the positivity of the energy of the weighted Paneitz operator of the adapted smooth metric measure space.  This is guaranteed by the following result, which is the weighted analogue of the positivity results for the Paneitz operator established in~\cite{Gursky1999,GurskyMalchiodi2014,XuYang2001}.

For convenience, we denote $m:=3-2\gamma$ in the rest of this section.

\begin{prop}
\label{prop:gursky}
Let $(X^{n+1},M^n,g_+)$ be a Poincar\'e--Einstein manifold and let $\gamma\in(1,2)$ if $n\geq4$ and $\gamma\in(1,3/2]$ if $n=3$.  Suppose that $h$ is a representative of the conformal boundary with positive scalar curvature and nonnegative fractional $Q$-curvature $Q_{2\gamma}$.  Set $m=3-2\gamma$ and let $(X^{n+1},g,y^m\dvol,m-1)$ be the adapted smooth metric measure space.  Let $f\in C^\infty(M)$ and let $U$ be the solution to~\eqref{eqn:12_extension_y}.  Then
\begin{equation}
\label{eqn:gursky}
\int_X \left[ \left(\Delta_\phi U\right)^2 - \left(4P - (m+n-1)J_\phi^mg\right)(\nabla U,\nabla U) \right]y^m\dvol \geq 0,
\end{equation}
where all quantities are computed with respect to the adapted smooth metric measure space.  Moreover, equality holds in~\eqref{eqn:gursky} if and only if $U$ is constant.
\end{prop}

\begin{proof}

From Lemma~\ref{lem:pe_smms_formulae} and Proposition~\ref{prop:positive_scalar_curvature} we know that $J_\phi^m=\frac{2}{2-m}J>0$, while the vanishing of $Q_\phi^m$ implies that
\begin{equation}
\label{eqn:qphi_zero}
\frac{1}{2-m}\Delta_\phi J = -\lv P_0\rv^2 + \frac{(m+n-1)(n-m+3)}{(2-m)^2(n+1)}J^2
\end{equation}
for $P_0=P-\frac{J}{n+1}g$ the tracefree part of the Schouten tensor of $g$.  From the proof of Theorem~\ref{thm:12_case} we know that $L_{2,\phi_0}^{m_0}U=0$ and hence, by Lemma~\ref{lem:12_rhostar},
\begin{equation}
\label{eqn:U_eqn}
\Delta_\phi U = 2y^{-1}\lp\nabla U,\nabla y\rp .
\end{equation}

We shall prove~\eqref{eqn:gursky} by using the weighted Bochner formula and adapting~\cite[Lemma~3.1]{Gursky1999} to the weighted setting.  To that end, recall that the weighted Bochner formula states that
\[ \frac{1}{2}\Delta_\phi\lv\nabla w\rv^2 = \lv\nabla^2w\rv^2 + \lp\nabla w,\nabla\Delta_\phi w\rp + \Ric_\phi^m(\nabla w,\nabla w) + mv^{-2}\lp\nabla w,\nabla v\rp^2 \]
for any smooth metric measure space $(X^{n+1},g,v^m\dvol,\mu)$ and any $w\in C^3(X)$.  Applying this to the adapted smooth metric measure space with $w=U$ and using Lemma~\ref{lem:pe_smms_formulae} and~\eqref{eqn:U_eqn} yields
\begin{align*}
\frac{1}{2}\Delta_\phi\lv\nabla U\rv^2 & = \lv\nabla^2U\rv^2 + \lp\nabla U,\nabla\Delta_\phi U\rp + \frac{m}{4}\left(\Delta_\phi U\right)^2 \\
& \quad + (m+n-1)P(\nabla U,\nabla U) + \frac{2}{2-m}J\lv\nabla U\rv^2 .
\end{align*}
It follows from the asymptotics~\eqref{eqn:12_expansion} that there are no boundary terms when integrating by parts with respect to $y^m\dvol$, and hence
\begin{equation}
\label{eqn:bochner}
0 = \int_X \left[ \lv\nabla^2 U\rv^2 - \frac{4-m}{4}\left(\Delta_\phi U\right)^2 + (m+n-1)P(\nabla U,\nabla U) + \frac{2}{2-m}J\lv\nabla U\rv^2\right] y^m\dvol .
\end{equation}
In particular, this implies that
\begin{align*}
& \int_X \left[ \left(\Delta_\phi U\right)^2 - \left(4P - (m+n-1)J_\phi^mg\right)(\nabla U,\nabla U) \right]y^m\dvol \\
& = \int_X \left[ \frac{4}{m+n-1}\lv\nabla^2 U\rv^2 + \frac{2m+n-5}{m+n-1}\left(\Delta_\phi U\right)^2 + \frac{2(m+n-1)^2+8}{(2-m)(m+n-1)}J\lv\nabla U\rv^2\right]y^m\dvol,
\end{align*}
from which the conclusion readily follows when $2m+n-5\geq0$.

Suppose now that $2m+n-5<0$.  Using~\eqref{eqn:qphi_zero}, it follows that for any $a>0$,
\begin{equation}
\label{eqn:P0_pointwise}
\begin{split}
-2P_0(\nabla U,\nabla U) & \geq -2\sqrt{\frac{n}{n+1}}\lv P_0\rv\,\lv\nabla U\rv^2 \\
& \geq -\frac{a}{J}\lv P_0\rv^2\,\lv\nabla U\rv^2 - \frac{nJ}{(n+1)a}\lv\nabla U\rv^2 \\
& = \frac{a}{(2-m)J}\lv\nabla U\rv^2\Delta_\phi J - \left(\frac{(m+n-1)(n-m+3)a}{(2-m)^2(n+1)} + \frac{n}{(n+1)a}\right)J\lv\nabla U\rv^2 .
\end{split}
\end{equation}
Using~\eqref{eqn:12_J_asymptotics} and the Cauchy--Schwarz inequality, we see that
\begin{align*}
\int_X \left(\frac{1}{J}\lv\nabla U\rv^2\Delta_\phi J\right)y^m\dvol & = \int_X\left[-\frac{1}{J}\lp\nabla\lv\nabla U\rv^2,\nabla J\rp + \frac{\lv\nabla J\rv^2}{J^2}\lv\nabla U\rv^2\right]y^m\dvol \\
& \qquad - \int_M \lim_{y\to0}\left(\frac{\lv\nabla U\rv^2}{J}y^m\frac{\partial J}{\partial y}\right)\,\dvol \\
& \geq -\int_X\lv\nabla^2U\rv^2\,y^m\dvol + \frac{8\gamma(\gamma-1)^2}{d_\gamma}\int_M \frac{Q_{2\gamma}}{J_h}\lv\nabla f\rv^2\,\dvol .
\end{align*}
Inserting~\eqref{eqn:bochner} into the above display and using the assumption $Q_{2\gamma}\geq0$ yields
\begin{multline}
\label{eqn:cs_estimate}
\int_X \left(\frac{1}{J}{\lv\nabla U\rv^2}\Delta_\phi J\right)y^m\dvol \\ \geq \int_X\left[\frac{m-4}{4}\left(\Delta_\phi U\right)^2 + (m+n-1)P_0(\nabla U,\nabla U) + \frac{4n-(n-3)m-m^2}{(2-m)(n+1)}J\lv\nabla U\rv^2\right]y^m\dvol .
\end{multline}
Combining~\eqref{eqn:P0_pointwise} and~\eqref{eqn:cs_estimate} together with the choice $a=-\frac{2(2-m)}{2m+n-5}$ yields
\begin{align*}
& \int_X \left[ \left(\Delta_\phi U\right)^2 - \left(4P - (m+n-1)J_\phi^mg\right)(\nabla U,\nabla U) \right]y^m\dvol \\
& = \int_X \left[ \left(\Delta_\phi U\right)^2 - 4P_0(\nabla U,\nabla U) + \left(\frac{2(m+n-1)}{2-m}-\frac{4}{n+1}\right)J\lv\nabla U\rv^2\right]y^m\dvol \\
& \geq \frac{4(n+1)^2-32-n(n-5)^2-6m(n+1)(m+n-5)}{(4-m)(2-m)(n+1)}\int_X J\lv\nabla U\rv^2y^m\dvol .
\end{align*}
Recall now that $m\in(-1,1)$ satisfies $2m+n-5<0$ and $m+n-3\geq0$ --- the latter inequality is equivalent to $\gamma\leq\frac{n}{2}$.  It is straightforward to check that the coefficient of the last integral is positive for $m$ satisfying these constraints, from which the conclusion immediately follows.
\end{proof}

We now have the ingredients to proof Theorem~\ref{thm:positive}, which we restate here for convenience.

\begin{thm}
Let $(X^{n+1},M^n,g_+)$ be a Poincar\'e--Einstein manifold, let $\gamma\in(1,2)$ if $n\geq 4$ and let $\gamma\in(1,3/2]$ if $n=3$.  Suppose that there is a representative $h$ of the conformal boundary with nonnegative scalar curvature $R_h$ and nonnegative fractional $Q$-curvature $Q_{2\gamma}$.  Then $P_{2\gamma}\geq0$.  Moreover, $\ker P_{2\gamma}\not=\{0\}$ if and only if $Q_{2\gamma}\equiv0$ or $n=2\gamma=3$, in which case $\ker P_{2\gamma}=\bR$ consists of the constant functions.
\end{thm}

\begin{proof}

Let $y$ be the adapted defining function.  Given $f\in C^\infty(M)$, it follows from Lemma~\ref{lem:12_rhostar_extension} that
\begin{multline}
\label{eqn:rhostar_energy_identity}
\int_M P_{2\gamma}f\,f\,\dvol_h = \frac{n-2\gamma}{2}\int_M Q_{2\gamma}f^2\,\dvol_h \\ + \frac{d_\gamma}{8\gamma(\gamma-1)}\int_X\left[\left(\Delta_\phi U\right)^2 - \left(4P-(m+n-1)J_\phi^mg\right)(\nabla U,\nabla U)\right]y^m\dvol_{g}
\end{multline}
for $U$ the solution to~\eqref{eqn:12_extension_y}.  From Corollary~\ref{cor:signs_of_r_and_q2}, we know that $R_h$ is semi-positive, and hence, by Lemma~\ref{lem:12_maximum}, in fact $R_h>0$.  We may then apply Proposition~\ref{prop:gursky} to conclude that $P_{2\gamma}\geq0$.  Moreover, for $f\not\equiv0$, Proposition~\ref{prop:gursky} implies $\int_M f\,P_{2\gamma}f=0$ if and only if $f$ is constant and either $Q_{2\gamma}=0$ or $n=2\gamma=3$, as desired.
\end{proof}

We now turn to the proof of Theorem~\ref{thm:maximum}, which gives necessary conditions for the fractional GJMS operators to satisfy a strong maximum principle.  The main technical ingredient in the proof is Lemma~\ref{lem:12_maximum}, after which point one can argue as in~\cite[Theorem~A]{GurskyMalchiodi2014}.  For the convenience of the reader, we restate Theorem~\ref{thm:maximum} here and provide a sketch of the proof.

\begin{thm}
\label{thm:12_maximum}
Let $(X^{n+1},M^n,g_+)$ be a Poincar\'e--Einstein manifold and suppose that there is a representative $h$ for the conformal boundary with scalar curvature $R_h\geq0$ and semi-positive fractional $Q$-curvature $Q_{2\gamma,h}$ for some $1<\gamma<\min\{2,n/2\}$ fixed.  Then for any $f\in C^\infty(M)$ such that $P_{2\gamma}f\geq0$, either $f>0$ or $f\equiv0$.  Moreover, if $f>0$, then the representative $h_f:=f^{\frac{4}{n-2\gamma}}h$ of the conformal boundary has positive scalar curvature and nonnegative fractional $Q$-curvature $Q_{2\gamma,h_f}$.
\end{thm}

\begin{proof}

If $\min f>0$ we are done, so suppose $\min f\leq 0$.  For $\lambda\in[0,1]$, define $f_\lambda:=1-\lambda+\lambda f$, so that $f_0\equiv1$ and $f_1=f$.  Set $\lambda_0=\frac{1}{1-\min f}$, so that $\min f_{\lambda_0}=0$ and $f_\lambda>0$ for $\lambda\in[0,\lambda_0)$.  Let $\lambda\in[0,\lambda_0)$ and set $h_\lambda=f_\lambda^{\frac{4}{n-2\gamma}}h$.  It follows from the definition of the fractional $Q$-curvatures and the conformal covariance of the fractional GJMS operators that
\[ \frac{n-2\gamma}{2}Q_{2\gamma,h_\lambda} = f_\lambda^{-\frac{n+2\gamma}{n-2\gamma}}\left(\frac{(n-2\gamma)(1-\lambda)}{2}Q_{2\gamma,h} + \lambda P_{2\gamma}f\right), \]
and hence $Q_{2\gamma,h_\lambda}$ is semi-positive.  By Lemma~\ref{lem:12_maximum}, we have that $R_{h}>0$.  Applying Lemma~\ref{lem:12_maximum} again, there cannot be a $\lambda_1\in[0,\lambda_0)$ such that $\min R_{h_{\lambda_1}}=0$, and hence a simple continuity argument shows that $R_{h_\lambda}>0$.  On the other hand, we directly compute that
\[ R_{h_\lambda} = f_\lambda^{-\frac{n-2\gamma+4}{n-2\gamma}}\left(-\frac{4(n-1)}{n-2\gamma}\Delta f - \frac{8(n-1)(\gamma-1)}{(n-2\gamma)^2}f^{-1}\lv\nabla u\rv^2 + R_hf\right) . \]
Combining these observations and taking the limit $\lambda\to\lambda_0$ yields
\begin{equation}
\label{eqn:12_maximum_obs}
\Delta f_{\lambda_0} \leq \frac{n-2\gamma}{4(n-1)}R_hf_{\lambda_0} .
\end{equation}
From the strong maximum principle we conclude that $f_{\lambda_0}\equiv0$.  If $\lambda_0=1$ then $f\equiv0$, and so we are done.  Otherwise $f=-\frac{1-\lambda_0}{\lambda_0}$, and hence
\[ 0 \leq P_{2\gamma}f = -\frac{(n-2\gamma)(1-\lambda_0)}{2\lambda_0}Q_{2\gamma,h}, \]
a contradiction.

Finally, the above argument showing that $R_{h_\lambda}>0$ for $\lambda\in(0,\lambda_0)$ implies that the scalar curvature of $h_f$ is nonnegative, while the assumption $P_{2\gamma}f\geq0$ implies that $Q_{2\gamma,h_f}$ is nonnegative.  From the argument of Lemma~\ref{lem:12_maximum} we see that either $R_{h_f}>0$ or $R_{h_f}\equiv0$.  The latter case contradicts that $R_{h}>0$, as $R_{h_f}\equiv0$ implies that the Yamabe constant of the conformal boundary is equal to zero, while $R_{h}>0$ implies that it is positive.
\end{proof}

Following the argument of~\cite[Proposition~2.4]{GurskyMalchiodi2014} but using Lemma~\ref{lem:12_maximum} and Theorem~\ref{thm:12_maximum} in place of~\cite[Lemma~2.1]{GurskyMalchiodi2014} and~\cite[Theorem~2.2]{GurskyMalchiodi2014}, respectively, we easily derive the positivity of the Green's function for the fractional GJMS operators $P_{2\gamma}$ under the assumptions of Theorem~\ref{thm:12_maximum}.

\begin{cor}
\label{cor:12_positive_g}
Let $(X^{n+1},M^n,g_+)$ be a Poincar\'e--Einstein manifold.  Let $\gamma\in(1,2)$ and suppose that there is a representative $h$ for the conformal boundary with $R\geq0$ and semi-positive $Q_{2\gamma}$.  Given $p\in M$, let $G_p$ denote the Green's function for $P_{2\gamma}$ with pole at $p$.  Then $G_p>0$ on $M\setminus\{p\}$.
\end{cor}

We expect that, through a synthesis of the ideas in~\cite{GonzalezQing2010,GurskyMalchiodi2014} together with our approach to studying the fractional GJMS operators $P_{2\gamma}$ with $\gamma\in(1,2)$ as boundary operators associated to weighted Paneitz operators, one can use these results to construct representatives of the conformal boundary with constant positive $Q_{2\gamma}$ under the assumptions of Theorem~\ref{thm:12_maximum}.

\appendix
\section{A factorization for the weighted GJMS operators}
\label{app:product}

In this appendix we prove the following factorization theorem for the GJMS operators of Riemannian products $(X^{n+1}\times F^m,g_+\oplus g_F)$ of Einstein manifolds $(X^{n+1},g_+)$ and $(F^m,g_F)$ with $\Ric(g_+)=-ng_+$ and $\Ric(g_F)=(m-1)g_F$ when acting on functions which are independent of $F$.  After treating $m$ as a formal variable, this yields Theorem~\ref{thm:weighted_gjms_factorization}.  Our proof is by an adaptation of the proof given by Fefferman and Graham of~\cite[Proposition~7.9]{FeffermanGraham2012} in the case $m=0$, which gives the factorization of the GJMS operators on an Einstein manifold.

\begin{thm}
\label{thm:product}
Let $(X^{n+1},g_+)$ and $(F^m,g_F)$ be two Einstein manifolds satisfying $\Ric(g_+)=-ng_+$ and $\Ric(g_F)=(m-1)g_F$ and let $k\in\bN$.  Denote by $\pi\colon X\times F\to X$ the canonical projection.  Then for any $f\in C^\infty(M)$, it holds that
\begin{equation}
\label{eqn:product}
L_{2k}\left(\pi^\ast f\right) = \left[\prod_{j=1}^k \left(-\Delta - \frac{(n-m+2k-4j+3)(m+n-2k+4j-3)}{4}\right)\right]\left(\pi^\ast f\right)
\end{equation}
where $L_{2k}$ is the $k$-th order GJMS operator of the metric $g_+\oplus h$.
\end{thm}

\begin{proof}

It was shown by Gover and Leitner~\cite{GoverLeitner2009} that the ambient metric of $(X^{n+1}\times F^m,g_+\oplus g_F)$ can be written in the form
\[ \tilde g = 2\rho\,dt^2 + 2t\,dt\,d\rho + t^2g_\rho \]
for
\begin{equation}
\label{eqn:grho}
g_\rho = \left(1-\frac{\rho}{2}\right)^2g_+ + \left(1+\frac{\rho}{2}\right)^2g_F .
\end{equation}
The GJMS operator $P_{2k}$~\cite{GJMS1992} is defined to be the obstruction to formally extending a given function $u_0\in C^\infty(X\times F)$ to a homogeneous function $U$ of degree $w=-\frac{n+m+1}{2}+k$ which is harmonic with respect to the ambient metric.  The extension $U=t^wu$ of $u_0$, where $u=u(x,\rho)$ and $x\in X\times F$, will be harmonic if and only if
\[ 0 = -2\rho u^{\prime\prime} + \left(2(k-1) - \rho\tr_{g_\rho}g_\rho^\prime\right)u^\prime + \Delta_\rho f - \frac{m+n+1-2k}{4}\tr_{g_\rho}g_\rho^\prime u , \]
where ``prime'' denotes differentiation with respect to $\rho$.  In the case $u_0=\pi^\ast f_0$, it follows that $u$ depends only on $X$ and $\rho$; i.e.\ we may write $u=\pi^\ast f$.  Moreover, using the specific form~\eqref{eqn:grho} of the metrics $g_\rho$ in the above display yields
\begin{align*}
0 & = -2\left(1-\frac{\rho}{2}\right)^2\left(1+\frac{\rho}{2}\right)\rho f^{\prime\prime} \\
& \quad + \left(2(k-1)\left(1-\frac{\rho}{2}\right)^2\left(1+\frac{\rho}{2}\right) + (n+1)\rho\left(1-\frac{\rho}{2}\right)\left(1+\frac{\rho}{2}\right) - m\rho\left(1-\frac{\rho}{2}\right)^2\right)f^\prime \\
& \quad + \left(\left(1+\frac{\rho}{2}\right)\Delta + \frac{(n+1)(m+n+1-2k)}{4}\left(1-\frac{\rho}{2}\right)\left(1+\frac{\rho}{2}\right) - \frac{m(m+n+1-2k)}{4}\left(1-\frac{\rho}{2}\right)^2\right)f .
\end{align*}
Differentiating $\ell$ times in $\rho$ and evaluating at $\rho=0$ yields
\begin{align*}
0 & = 2(k-\ell-1)f^{(\ell+1)} + \left(\Delta + \ell(\ell-k+n+1-m) + \frac{(n+1-m)(m+n+1-2k)}{4}\right)f^{(\ell)} \\
& \quad + \frac{\ell}{2}\left(\Delta + (\ell-1)(\ell-k-1+2m) + \frac{m(m+n+1-2k)}{2}\right)f^{(\ell-1)} \\
& \quad + \frac{\ell(\ell-1)}{4}\left((\ell-2)(k-\ell+1-m-n) - \frac{(m+n+1)(m+n+1-2k)}{4}\right)f^{\ell-2} .
\end{align*}
Denote by $(a)_k=a(a+t)\dotsm(a+k-1)=\frac{\Gamma(a+k)}{\Gamma(a)}$.  Multiplying the above equation by $\frac{1}{2}(k-\ell)_\ell$ gives the equivalent formula
\begin{align*}
0 & = (k-\ell-1)_{\ell+1}f^{(\ell+1)} \\
& \quad + \left(\frac{1}{2}\Delta - \frac{\ell(k-\ell)}{2} + \frac{(n+1-m)(m+n+1-2k+4\ell)}{8}\right)(k-\ell)_\ell f^{(\ell)} \\
& \quad + \frac{\ell(k-\ell)}{2}\left(\frac{1}{2}\Delta - \frac{(\ell-1)(k-\ell+1)}{2} + \frac{m(m+n+1-2k+4\ell-4)}{4}\right)(k-\ell+1)_{\ell-1}f^{(\ell-1)} \\
& \quad + \frac{(\ell-1)_2(k-\ell)_2}{8}\left((\ell-2)(k-\ell+2) - \frac{(m+n+1)(m+n+1-2k+4\ell-8)}{4}\right)(k-\ell+2)_{\ell-2}f^{(\ell-2)} .
\end{align*}
In other words, we see that
\begin{equation}
\label{eqn:formula}
(k-\ell)_\ell f^{(\ell)} = q_\ell f
\end{equation}
for $q_\ell$ defined by $q_{-2}=q_{-1}=0$, $q_0=1$, and
\begin{equation}
\label{eqn:q_recursive}
\begin{split}
q_{\ell+1} & = \left(-\frac{1}{2}\Delta + \frac{\ell(k-\ell)}{2} - \frac{(n+1-m)(m+n+1-2k+4\ell)}{8}\right)q_{\ell} \\
& \quad + \frac{\ell(k-\ell)}{2}\left(-\frac{1}{2}\Delta + \frac{(\ell-1)(k-\ell+1)}{2} - \frac{m(m+n+1-2k+4\ell-4)}{4}\right)q_{\ell-1} \\
& \quad + \frac{(\ell-1)_2(k-\ell)_2}{4}\left(-\frac{(\ell-2)(k-\ell+2)}{2} + \frac{(m+n+1)(m+n+1-2k+4\ell-8)}{8}\right)q_{\ell-2} .
\end{split}
\end{equation}
Via a straightforward but tedious induction argument, we find that
\begin{equation}
\label{eqn:q_formula}
\begin{split}
q_\ell & = \sum_{j=0}^\ell\bigg[\binom{\ell}{j}2^{j-\ell}(k-\ell)_{\ell-j}\left(\frac{m+n+1-2k+4j}{2}\right)_{\ell-j} \\
& \qquad \times \prod_{i=1}^{j}\left(-\frac{1}{2}\Delta - \frac{(n-m+2k-4i+3)(m+n-2k+4i-3)}{8}\right)\bigg] .
\end{split}
\end{equation}
On the other hand, \eqref{eqn:formula} implies that the obstruction to finding the harmonic extension $U$ of $u$ is
\[ q_k = \prod_{i=1}^{k}\left(-\frac{1}{2}\Delta - \frac{(n-m+2k-4i+3)(m+n-2k+4i-3)}{8}\right) . \]
Since $2^kq_k$ has leading order term $(-\Delta)^k$, this finishes the proof.
\end{proof}

\bibliographystyle{abbrv}
\bibliography{../bib}

\newcommand{\noopsort}[1]{}
\begin{thebibliography}{10}

\bibitem{AbramowitzStegun}
M.~Abramowitz and I.~A. Stegun.
\newblock {\em Handbook of mathematical functions with formulas, graphs, and
  mathematical tables}, volume~55 of {\em National Bureau of Standards Applied
  Mathematics Series}.
\newblock For sale by the Superintendent of Documents, U.S. Government Printing
  Office, Washington, D.C., 1964.

\bibitem{Besse}
A.~L. Besse.
\newblock {\em Einstein manifolds}, volume~10 of {\em Ergebnisse der Mathematik
  und ihrer Grenzgebiete (3)}.
\newblock Springer-Verlag, Berlin, 1987.

\bibitem{Biquard2008}
O.~Biquard.
\newblock Continuation unique \`a partir de l'infini conforme pour les
  m\'etriques d'{E}instein.
\newblock {\em Math. Res. Lett.}, 15(6):1091--1099, 2008.

\bibitem{CaffarelliSilvestre2007}
L.~Caffarelli and L.~Silvestre.
\newblock An extension problem related to the fractional {L}aplacian.
\newblock {\em Comm. Partial Differential Equations}, 32(7-9):1245--1260, 2007.

\bibitem{Case2011t}
J.~S. Case.
\newblock Smooth metric measure spaces, quasi-{E}instein metrics, and tractors.
\newblock {\em Cent. Eur. J. Math.}, 10(5):1733--1762, 2012.

\bibitem{ChangGonzalez2011}
S.-Y.~A. Chang and M.~d.~M. Gonz{\'a}lez.
\newblock Fractional {L}aplacian in conformal geometry.
\newblock {\em Adv. Math.}, 226(2):1410--1432, 2011.

\bibitem{ChangQingYang2006}
S.-Y.~A. Chang, J.~Qing, and P.~Yang.
\newblock On the renormalized volumes for conformally compact einstein
  manifolds.
\newblock {\em J.\ Math.\ Sciences (N.\ Y.)}, 149(6):1755--1769, 2008.

\bibitem{ChruscielDelayLeeSkinner2005}
P.~T. Chru{\'s}ciel, E.~Delay, J.~M. Lee, and D.~N. Skinner.
\newblock Boundary regularity of conformally compact {E}instein metrics.
\newblock {\em J. Differential Geom.}, 69(1):111--136, 2005.

\bibitem{Escobar1992a}
J.~F. Escobar.
\newblock Conformal deformation of a {R}iemannian metric to a scalar flat
  metric with constant mean curvature on the boundary.
\newblock {\em Ann. of Math. (2)}, 136(1):1--50, 1992.

\bibitem{Escobar1992}
J.~F. Escobar.
\newblock The {Y}amabe problem on manifolds with boundary.
\newblock {\em J. Differential Geom.}, 35(1):21--84, 1992.

\bibitem{FeffermanGraham2002}
C.~Fefferman and C.~R. Graham.
\newblock {$Q$}-curvature and {P}oincar\'e metrics.
\newblock {\em Math. Res. Lett.}, 9(2-3):139--151, 2002.

\bibitem{FeffermanGraham2012}
C.~Fefferman and C.~R. Graham.
\newblock {\em The ambient metric}, volume 178 of {\em Annals of Mathematics
  Studies}.
\newblock Princeton University Press, Princeton, NJ, 2012.

\bibitem{FeffermanGraham2013}
C.~Fefferman and C.~R. Graham.
\newblock Juhl's formulae for {GJMS} operators and {$Q$}-curvatures.
\newblock {\em J. Amer. Math. Soc.}, 26(4):1191--1207, 2013.

\bibitem{FischerColbrieSchoen1980}
D.~Fischer-Colbrie and R.~Schoen.
\newblock The structure of complete stable minimal surfaces in {$3$}-manifolds
  of nonnegative scalar curvature.
\newblock {\em Comm. Pure Appl. Math.}, 33(2):199--211, 1980.

\bibitem{GonzalezQing2010}
M.~d.~M. Gonz{\'a}lez and J.~Qing.
\newblock Fractional conformal {L}aplacians and fractional {Y}amabe problems.
\newblock {\em Anal. PDE}, 6(7):1535--1576, 2013.

\bibitem{Gover2006q}
A.~R. Gover.
\newblock Laplacian operators and {$Q$}-curvature on conformally {E}instein
  manifolds.
\newblock {\em Math. Ann.}, 336(2):311--334, 2006.

\bibitem{GoverLeitner2009}
A.~R. Gover and F.~Leitner.
\newblock A sub-product construction of {P}oincar\'e-{E}instein metrics.
\newblock {\em Internat. J. Math.}, 20(10):1263--1287, 2009.

\bibitem{Graham2000}
C.~R. Graham.
\newblock Volume and area renormalizations for conformally compact {E}instein
  metrics.
\newblock In {\em The {P}roceedings of the 19th {W}inter {S}chool ``{G}eometry
  and {P}hysics'' ({S}rn\'\i, 1999)}, number~63, pages 31--42, 2000.

\bibitem{GrahamHirachi2004}
C.~R. Graham and K.~Hirachi.
\newblock The ambient obstruction tensor and {$Q$}-curvature.
\newblock In {\em Ad{S}/{CFT} correspondence: {E}instein metrics and their
  conformal boundaries}, volume~8 of {\em IRMA Lect. Math. Theor. Phys.}, pages
  59--71. Eur. Math. Soc., Z\"urich, 2005.

\bibitem{GJMS1992}
C.~R. Graham, R.~Jenne, L.~J. Mason, and G.~A.~J. Sparling.
\newblock Conformally invariant powers of the {L}aplacian. {I}. {E}xistence.
\newblock {\em J. London Math. Soc. (2)}, 46(3):557--565, 1992.

\bibitem{GrahamLee1991}
C.~R. Graham and J.~M. Lee.
\newblock Einstein metrics with prescribed conformal infinity on the ball.
\newblock {\em Adv. Math.}, 87(2):186--225, 1991.

\bibitem{GrahamZworski2003}
C.~R. Graham and M.~Zworski.
\newblock Scattering matrix in conformal geometry.
\newblock {\em Invent. Math.}, 152(1):89--118, 2003.

\bibitem{GuillarmouGuillope2007}
C.~Guillarmou and L.~Guillop{\'e}.
\newblock The determinant of the {D}irichlet-to-{N}eumann map for surfaces with
  boundary.
\newblock {\em Int. Math. Res. Not. IMRN}, (22):Art. ID rnm099, 26, 2007.

\bibitem{GuillarmouQing2010}
C.~Guillarmou and J.~Qing.
\newblock Spectral characterization of {P}oincar\'e-{E}instein manifolds with
  infinity of positive {Y}amabe type.
\newblock {\em Int. Math. Res. Not. IMRN}, (9):1720--1740, 2010.

\bibitem{Gursky1999}
M.~J. Gursky.
\newblock The principal eigenvalue of a conformally invariant differential
  operator, with an application to semilinear elliptic {PDE}.
\newblock {\em Comm. Math. Phys.}, 207(1):131--143, 1999.

\bibitem{GurskyMalchiodi2014}
M.~J. Gursky and A.~Malchiodi.
\newblock A strong maximum principle for the {P}aneitz operator and a non-local
  flow for the {$Q$}-curvature.
\newblock arXiv:1401.3216, \noopsort{2099}Preprint.

\bibitem{HangYang2014}
F.~Hang and P.~C. Yang.
\newblock Sign of {G}reen's function of {P}aneitz operators and the ${Q}$
  curvature.
\newblock arXiv:1411.3924, \noopsort{2099}Preprint.

\bibitem{Juhl2013}
A.~Juhl.
\newblock Explicit formulas for {GJMS}-operators and {$Q$}-curvatures.
\newblock {\em Geom. Funct. Anal.}, 23(4):1278--1370, 2013.

\bibitem{Lee1995}
J.~M. Lee.
\newblock The spectrum of an asymptotically hyperbolic {E}instein manifold.
\newblock {\em Comm. Anal. Geom.}, 3(1-2):253--271, 1995.

\bibitem{MazzeoMelrose1987}
R.~R. Mazzeo and R.~B. Melrose.
\newblock Meromorphic extension of the resolvent on complete spaces with
  asymptotically constant negative curvature.
\newblock {\em J. Funct. Anal.}, 75(2):260--310, 1987.

\bibitem{ONeill}
B.~O'Neill.
\newblock {\em Semi-{R}iemannian geometry}, volume 103 of {\em Pure and Applied
  Mathematics}.
\newblock Academic Press Inc. [Harcourt Brace Jovanovich Publishers], New York,
  1983.
\newblock With applications to relativity.

\bibitem{Qing2003}
J.~Qing.
\newblock On the rigidity for conformally compact {E}instein manifolds.
\newblock {\em Int. Math. Res. Not.}, (21):1141--1153, 2003.

\bibitem{XuYang2001}
X.~Xu and P.~C. Yang.
\newblock Positivity of {P}aneitz operators.
\newblock {\em Discrete Contin. Dynam. Systems}, 7(2):329--342, 2001.

\bibitem{Yang2013}
R.~Yang.
\newblock On higher order extensions for the fractional {L}aplacian.
\newblock arXiv:1302.4413, \noopsort{2099}Preprint.

\end{thebibliography}
\end{document}